\documentclass[12pt]{amsart}%
\usepackage{amsfonts}
\usepackage{amsmath}
\usepackage{amssymb}
\usepackage{amsthm}
\usepackage{mathrsfs}
\usepackage{graphicx}%
\usepackage{amsmath}
\usepackage{setspace}
\usepackage{hyperref}
\usepackage{cite}
\usepackage{amsfonts}
\usepackage{algorithmic}
\usepackage{textcomp}
\newcommand{\R}{\mathbb R}
\allowdisplaybreaks[4]
\makeatletter
\@addtoreset{equation}{section}
\makeatother

\newtheorem{theorem}{Theorem}[section]

\newtheorem{definition}[theorem]{Definition}

\newtheorem{lemma}[theorem]{Lemma}

\newtheorem{proposition}[theorem]{Proposition}

\makeatletter
\def\@makefnmark{}
\makeatother

\begin{document}

\title[$C^{\infty}$ regularity of the Alt-Phillips Functional  for negative powers]{$C^{\infty}$ regularity of the Alt-Phillips Functional  for negative powers}

\author{Lu Chen}
\address[Lu Chen]{Key Laboratory of Algebraic Lie Theory and Analysis of Ministry of Education, School of Mathematics and Statistics, Beijing Institute of Technology, Beijing
100081, PR China;  Tangshan Research Institute, Beijing Institute of Technology, Tangshan 063000, PR China}
\email{chenlu5818804@163.com}

\author{Jiali Lan}
\address[Jiali Lan]{Key Laboratory of Algebraic Lie Theory and Analysis of Ministry of Education, School of Mathematics and Statistics, Beijing Institute of Technology, Beijing
100081, PR China}
\email{17636268505@163.com}

\author{Yong Wu}
\address[Yong Wu]{Key Laboratory of Algebraic Lie Theory and Analysis of Ministry of Education, School of Mathematics and Statistics, Beijing Institute of Technology, Beijing
100081, PR China}
\email{3120251484@bit.edu.cn}

\address{}

%\thanks{$*$ Corresponding author.}
\thanks{The first author was partly supported by the  National Natural Science Foundation of China (No. 12271027) and Hebei Natural
Science Foundation (No. A2025105003). }

\begin{abstract}
 In this paper, we study  the regularity of the free boundary for minimizers of the Alt-Phillips functional with negative powers
\[\mathcal{E}_{\gamma}(u)=\int_{\Omega}\frac{1}{2}|\nabla u|^2+\frac{1}{\gamma}u^{-\gamma}\chi_{\{u>0\}}dx,\quad\gamma\in(0,2).\]
We proved that the free boundaries are $C^{\infty}$ at regular points. A key technical tool is the linearized operator for the PDE satisfied by the partial derivatives of a solution  to the Alt-Phillips Euler-Lagrange equation in the negative power case. For this operator we establish a comparison principle, which may have further applications to the Alt-Phillips problem with negative powers.
\end{abstract}

\maketitle

 \section{Introduction and Main Results}
 One fundamental problem  in the calculus of variations consists in studying the minimizers of energy functionals
 \[J(u,\Omega):=\int_{\Omega}\frac{1}{2}|\nabla u|^2+ W(u)dx,\]
  prescribed with the boundary condition
  \[u=\varphi\geq0\quad\text{on }\partial\Omega.\]
 The potential $W:\mathbb{R}\rightarrow[0,\infty)$  is usually assumed to be non-decreasing with  minimum 0, which guarantee that minimizers of $J$ are non-negative and satisfy the following Euler-Lagrange equation
\begin{equation}\label{e1.5}\Delta u=W'(u)\quad\text{in }\Omega.\end{equation}
From the strong maximum principle, the solutions of \eqref{e1.5} must be positive strictly in the interior of $\Omega$ whenever $W$ is of class $C^{1,1}$ at the origin.  An important example is the Allen-Cahn  energy given by the double-well potential
 \[W(t)=(1-t^2)^2,\]
 which is relevant in the theory of phase-transitions and minimal surfaces. In \cite{MM}, Modica and Mortola showed that $0$-homogenous rescalings
of bounded minimizers $|u|\leq 1$  converge up to subsequences to a $\pm1$ configuration separated by a minimal surface, i.e.
 \[u_{\varepsilon}(x)=u\left(\frac{x}{\varepsilon}\right)\rightarrow \chi_{E}-\chi_{E^c}\quad\text{in }L_{loc}^1\]
 as $\varepsilon\rightarrow0$, with $E$ a set of minimal perimeter. At the level of the energy, this result is expressed in terms of the $\Gamma$-convergence of the rescaled energies
\[\int_{\Omega}\varepsilon |\nabla u|^2+\frac{1}{\varepsilon}W(u)dx\]
 to a multiple of the perimeter functional $c_0 P_{\Omega}(E)$. The Modica-Mortola Theorem is a significant result in the calculus of variations with numerous applications, particularly in the study of phase transitions. It provides a key prediction for physical reality: within a rigid container, two immiscible fluids will separate so as to minimize the area of their interface.

\medskip
 In \cite{S}, Savin proved that if the $0$ level set of the minimizers of the Allen-Cahn energy is included in a flat cylinder then, in the interior, it is included in a flatter cylinder. As a consequence, he confirmed the De Giorgi conjecture  which states that level sets of global solutions $u$ are hyperplanes in dimension $n\leq8$ under the following condition:
 \[|u|\leq 1\quad \partial_n u>0\quad \lim_{x\rightarrow\pm\infty}u(x',x_n)=\pm1.\]
This result reveals a deep connection between the theory of phase transitions and the theory of minimal surfaces.
The De Giorgi conjecture in the whole space $\mathbb{R}^n$  has also been  studied. It was first proved in \cite{GG} for $n=2$ and later for $n=3$ in \cite{AC3}, while the case  $n=4$ remains open.

\medskip
On the other hand,  when $W\notin C^{1,1}$  near one of its minimum points,  then a minimizer can develop constant patches  and this leads to a free boundary problem.   In other words, there are two regions in $\Omega$:  one in which $\{u=0\}$ and one in which $\{u>0\}$. Moreover, if we denote
 \[F(u):=\partial\{u>0\}\cap \Omega,\]
 then this is called the free boundary, which splits into regular set and singular set:
 \[F(u)=\mathrm{Reg}(F(u))\cup \mathrm{Sing}(F(u)).\]
The set of regular points is an open subset of the free boundary, while the singular points are those at which the contact set $\{u=0\}$ has density zero. One of the most important and interesting problem is the regularity of the free boundaries which is intimately connected to the classification of global minimizers for $J$.  Two such particular potentials  were investigated systematically.

\medskip
 The first one is the Lipschitz potential
\[W(t)=t^+,\]
 which corresponds to the classical obstacle problem.
 The $C^{1,1}$ regularity of $u$ when $\varphi$ is $C^{1,1}$ was first obtained by Frehse in \cite{F}. Later, Caffarelli  \cite{C3} proved that global solutions to the obstacle problem are convex. Using the monotonicity formula from \cite{ACF}, he established for the first time the regularity of free boundaries  outside a certain set of singular points.  For more comprehensive survey, we refer the reader to the book of Petrosyan, Shahgholian and Uraltseva \cite{PSU} for an introduction to this subject.

\medskip
Another one is the discontinuous potential
 \[W(t)=\chi_{\{t>0\}},\]
with its associated Alt-Caffarelli energy,   which is known as the one-phase or Bernoulli free boundary problem.
In \cite{AC}, Alt and Caffarelli established the existence of a minimizer and proved that the reduced free boundary is a regular surface.
Later Caffarelli developed a Harnack inequality approach to the regularity of free boundaries, see \cite{C1,C4,C5} for more details.
Using a different approach from Caffarelli's classical supconvolution method, DE Silva \cite{S}  proved that Lipschitz free boundaries are $C^{1,\alpha}$. The key idea is to establish an "improvement of flatness" property for the graph of $u$ in the one-phase problem using a Harnack-type inequality. For a comprehensive account of the basic free boundary theory, we refer to Caffarelli and Salsa \cite{CS}.

\medskip
These two important examples are part of the more general family of Alt-Phillips potentials,  which correspond to the power-growth potentials
\begin{equation}\label{e1.6}W(t)=(t^+)^{\gamma}\quad\text{with }\gamma\in(0,2), \quad\text{and }\Delta u=\gamma u^{\gamma-1}.\end{equation}
Via the maximum principle, the classical solutions exist and  they are nonnegative. For $1<\gamma<2$, such an equation is used for modelling the distribution of a gas with density $u(x)$, in reaction with a porous catalyst pellet $\Omega$ \cite{A}. When $0<\gamma\leq1$, these potentials interpolate between the one-phase problem $\gamma=0$ and the obstacle problem $\gamma=1$.

\medskip
Take a simple change of variables
\[w=u^{\frac{1}{\alpha}},\quad\alpha=\frac{2}{2-\gamma}\text{ with }\gamma\in(0,2),\]
then $w$ satisfies   a degenerate equation
\[\Delta w=\frac{h(\nabla w)}{w}\quad\text{in }\{w>0\},\]
with
\[h(\nabla w)=0\quad\text{on } F(w),\]
where $h(p)=\frac{\gamma}{\alpha}-(\alpha-1)|p|^2$.
In an average sense, Alt and Phillips  \cite{AP} proved that $w$ is Lipschitz continuous and must grow like the distance to the free boundary. They further established the smoothness of the reduced part of the free boundary. DE Silva and Savin  \cite{SS}  developed the viscosity theory for the Alt-Phillips functional for general functions $h$ and proved an existence and regularity theory. The sign assumptions on the function $h$ are crucial, when $h<0$ in $B_1$,  the problem becomes completely different and it would correspond to the case of negative power potentials:
\begin{equation}\label{e1.7}W(t)=\frac{1}{\gamma}t^{-\gamma}\chi_{\{t>0\}},\quad\gamma\in(0,2).\end{equation}
These potentials are relevant in liquid models with large cohesive internal forces in regions of low density. The upper bound $\gamma<2$  is necessary for the finiteness of the energy. As $\gamma\rightarrow 2$,  the energy concentrates more and more near the free boundary,
and heuristically, the free boundary should minimize the surface area in the limit.

\medskip
In this paper, we consider the minimization problem of the Alt-Phillips functional involving negative power potentials
\begin{equation}\label{e1.1}\mathcal{E}_{\gamma}(u)=\int_{\Omega}\frac{1}{2}|\nabla u|^2+\frac{1}{\gamma}u^{-\gamma}\chi_{\{u>0\}}dx,\quad\gamma\in(0,2),\end{equation}
%%%补充相关研究进展
In \cite{DS2}, D. DE Silva and O. Savin showed that the non-negative minimizer $u$ of $\mathcal{E}_{\gamma}$ is a viscosity solution to the following degenerate one-phase free boundary problem:
\begin{equation}\label{e1.2}\begin{cases}
\Delta u=-u^{-\gamma-1}&\text{in }\{u>0\}\cap B_1\\
u(x_0+t\nu)=c_0 t^{\alpha}+o(t^{2-\alpha})&\text{on }F(u):=\partial\{u>0\}\cap B_1,
\end{cases}\end{equation}
with $t\geq0$, $\nu$ the unit normal to $F(u)$ at $x_0$ pointing towards $\{u>0\}$ and
\begin{equation}\label{e1.3}
\alpha:=\frac{2}{2+\gamma}\quad c_0:=[\alpha(1-\alpha)]^{-\frac{1}{2+\gamma}}.
\end{equation}
using a monotonicity formula and dimension reduction, they  proved that the free boundaries of the non-negative minimizers of $\mathcal{E}_{\gamma}$ are $C^{1,\delta_0}$ surfaces for some $\delta_0\in(0,1)$   except on a closed singular set $\Sigma_{u}\subset F(u)$ of Hausdorff dimension $n-3$. They further obtained density estimates for the free boundaries in \cite{DS3}, and established a $\Gamma$-convergence result of  $\mathcal{E}_{\gamma}$ (properly rescaled) to the Dirichlet-perimeter functional
\[\int_{\Omega}|\nabla u|^2dx+\mathrm{Per}_{\Omega}(\{u=0\}). \]
They also  investigated the rigidity of non-negative global minimizers of $\mathcal{E}_{\gamma}$ and  showed that the global minimizers in $\mathbb{R}^n$ are one-dimensional if $\gamma$ is close to $2$ and $n\leq7$, or if  $\gamma$  is close to $0$ and $n\leq4$.

\medskip
Very recently, the higher regularity for free boundaries in the Alt-Phillips problem with positive power potentials \eqref{e1.6} has also been studied. Fusch and Koch \cite{FK} proved it for $\gamma\in(1,2)$  and Allen et al. \cite{AKS} for $\gamma\in(0,\frac{2}{3}]$. Restrepo and Ros-Oton \cite{RRO} using a completely different method and proved that the set of regular points is actually $C^{\infty}$ for any $\gamma\in(0,2)$.
Inspired by the method of \cite{RRO},  this paper investigates the higher regularity of the Alt-Phillips functional with negative power potentials \eqref{e1.7} with $\gamma\in(0,2)$.

When we completed this work, we learned that Carducci and Tortone \cite{CT} had independently proved the $C^{\infty}$ regularity of minimizers for the Alt-Phillips functional for all exponents $\gamma\in(-2,2)$. Their approach is based on a hodograph transformation near regular points (see the seminal work of Kinderlehrer and Nirenberg \cite{KN}), which reformulates the problem as establishing Schauder estimates for a degenerate quasilinear PDE with Neumann boundary conditions. In contrast, the strategy we employ follows that of Restrepo and Ros-Oton \cite{RRO}, which focuses on the PDE satisfied by the ratio of partial derivatives of a solution (see \cite{DS4,DS5}, where this method was originally introduced in the context of the Alt-Caffarelli problem and the obstacle problem).

%In this paper, we introduce a linearized operator (for the negative power case)   satisfied by the partial derivatives of a solution  to the first equation in \eqref{e1.5}.

\medskip
We first recall the notation of viscosity solution to \eqref{e1.2} introduced in \cite{DS2}. We say that a continuous function $u$ touches a continuous function $\phi$ by above (resp. by below) at a point $x_0$ if
\[u\geq\phi\ (resp. u\leq\phi)\text{ in }B_r(x_0)\text{ with }u(x_0)=\phi(x_0).\]
We now consider the class  $\mathcal{C}^{+}$ of continuous functions $\phi$ vanishing on $\partial B_R(x_0)$   such that $\phi(x)=\phi(|x-x_0|)$ positive in $B_R(x_0)$ and extend it to be $0$ outside $B_R(x_0)$. Denote the distance function $d_1(x):=dist(x,\partial B_R(x_0))$ for $x\in B_R(x_0)$ and $0$ otherwise. Similarly, define the  class  $\mathcal{C}^{-}$ with $\phi$ positive outside $B_R(x_0)$ and $\phi=0$ in $B_R(x_0)$. The distance function $d_1(x):=dist(x,\partial B_R(x_0))$ for $x\in B_R^c(x_0)$ and $0$ otherwise.

\begin{definition}\label{definition1.1}(\cite{DS2}) A non-negative function $u$ is the viscosity solution for \eqref{e1.2} if
\begin{itemize}
\item[(i)] $u$ is $C^{\infty}$ and satisfies the equation in the classical sense in $\{u>0\}\cap B_1$;
\item[(ii)] if $x_0\in F(u)$, then $u$ cannot touch $\psi\in\mathcal{C}^+$ (resp. $\mathcal{C}^-$) by above (resp. below) at $x_0$, with
\[\psi(x):=c_0d_1(x)^{\alpha}+\mu d_1(x)^{2-\alpha},\]
for some $\mu>0$ (resp. $\mu<0$).
\end{itemize}
\end{definition}

Our main result can be stated as follows.
\begin{theorem}\label{theorem1.1}
Let $u$ be a non-negative viscosity solution of problem \eqref{e1.2}. Assume that $x_0\in F(u)$ is a regular free boundary point. Then,
\begin{itemize}
\item[(i)] the free boundary $F(u):=\partial\{u>0\}\cap B_1$ is $C^{\infty}$ in a neighborhood of $x_0$;
\item[(ii)] the functions $u/d^{\frac{2}{2+\gamma}}$ and $u^{\frac{2+\gamma}{2}}$ are $C^{\infty}$ in a neighborhood of $x_0$, where $d(x)=dist(x,F(u))$ is the distance function to free boundary.
\end{itemize}
\end{theorem}

The idea is to show that $u_i d^{\frac{\gamma}{2+\gamma}}$, $u/d^{\frac{2}{2+\gamma}}$ and $u_i/u_n$ are $C^{\infty}$, namely,
\begin{equation}\label{e1.8}
\partial\Omega\in C^{k+1,\delta}\quad \stackrel{Step 1}{\Longrightarrow}\quad u_i d^{\frac{\gamma}{2+\gamma}},\ \frac{u}{d^{\frac{2}{2+\gamma}}}\in C^{k,\delta}\quad\stackrel{Step 2}{\Longrightarrow}\quad\frac{u_i}{u_n}\in C^{k+1,\delta}\quad\Longrightarrow\quad \partial\Omega\in C^{k+2,\delta}.
\end{equation}
Step 1 of \eqref{e1.8} will be proved inductively in Proposition \ref{proposition4.1}. To prove Step 2, we observe that the quotient $w:=\frac{u_i}{u_n}$ satisfies
\[\mathrm{div}(u_n^2 w)=0\quad\text{in}\quad\Omega\cap B_1.\]
Using Step 1, the equation can be rewritten as
\[\mathrm{div}(d^s(x)a(x)\nabla w(x))=0\quad\text{in}\quad\Omega\cap B_1,\]
where $s=-\frac{2\gamma}{2+\gamma}>-1$ and $a(x)\in C^{k,\delta}$. Therefore, we can apply  Theorem 1.1 in \cite{TTV} to deduce that $w\in C^{k+1,\delta}$, which  in turn  implies that  $\partial\Omega\in C^{k+2,\delta}$.

\medskip
%The main method employed in this paper is adapted from \cite{RRO}.
 Recall that for the Alt-Phillips problem with positive powers \eqref{e1.6}, solutions $u$ behave like $d^{\frac{2}{2-\gamma}}$ at regular points, and consequently their derivatives $u_i$ behave like $d^{\frac{\gamma}{2-\gamma}}$ for any $\gamma\in(0,2)$. Given this behavior,  the authors introduced the linearized operator
\[L_{\kappa_{\gamma}}(v):=-\Delta v+\kappa_{\gamma}\frac{v}{d^2},\]
where $\kappa_{\gamma}:=\frac{\gamma}{2-\gamma}(\frac{\gamma}{2-\gamma}-1)$. Note that $\kappa_{\gamma} \geq -\frac{1}{4}$, and in one dimension, the operator is characterized by two exponents:
\[\theta_-=\frac{1-\sqrt{1+4\kappa_{\gamma}}}{2}\quad\text{and}\quad \theta_+=\frac{1+\sqrt{1+4\kappa_{\gamma}}}{2}.\]
This means that there are two completely different cases: for $\gamma \geq \frac{2}{3}$, the exponents $\theta_+=\frac{\gamma}{2-\gamma}$ and $\theta_-=1-\frac{\gamma}{2-\gamma}$, while for $\gamma \leq \frac{2}{3}$,  the exponents $\theta_-=\frac{\gamma}{2-\gamma}$ and $\theta_+=1-\frac{\gamma}{2-\gamma}$.

\medskip
In this paper, we consider the Alt-Phillips problem with negative powers \eqref{e1.7}. In this setting, solutions $u$ behave like $d^{\frac{2}{2+\gamma}}$ at regular points, and thus $u_i$ behave like $d^{-\frac{\gamma}{2+\gamma}}$. We similarly define a linearized operator
\[L_{\kappa}(v)=-\Delta v+\kappa\frac{v}{d^2},\]
with $\kappa=\frac{\gamma}{2+\gamma}(\frac{\gamma}{2+\gamma}-1)$. For this operator we establish a comparison principle (discussed in Section 3), which may also have other applications to the Alt-Phillips problem with negative powers.

Additionally, we prove a $C^{1,\delta}$ regularity at regular points for any $\delta\in(0,1)$ in Section 2, which is of independent interest. The $C^{1,\delta}$  regularity would allow one to derive a local expansion  near a regular point, which could then be used to classify monotone minimizing cones for the Alt-Phillips problem with negative exponent, a question we are interested in.

A key different from the positive power case lies in the characteristic exponents of the operator. For positive powers,  the specific form of the exponents $\theta_\pm$  depending on whether $\gamma$ is greater than or less than $2/3$. In contrast, for the negative powers, the two characteristic exponents in one dimension are given by the unified expressions
\[a_+=1+\frac{\gamma}{2+\gamma}\quad\text{and}\quad a_-=-\frac{\gamma}{2+\gamma}\]
for any $\gamma\in(0,2)$.  It follows that the derivatives $u_i$ behave as  $d^{a_-}$. Our analysis relies on understanding the one-dimensional solutions of $L_{\kappa}v=h d^{\frac{2}{2+\gamma}-2}$. Here  $a_-\in(-1/2,0)$, whereas in the positive exponent case  $\theta_-\in(-\infty,1/2)$. This difference leads to distinct growth behavior of $v$ near the free boundary.
 %However, since $a_- < 0$, the estimates in \cite{RRO} are no longer directly applicable. Fortunately, the range $\left(-\frac{1}{2}, 0\right)$ of the exponent  $a_-$ in the negative power case makes it possible to modify the method from \cite{RRO} for our setting. Further details are provided in Section 3.

\medskip
This paper is organized as follows. In Section 2, we  improve the $C^{1,\delta_0}$ regularity of the free boundary to $C^{1,\delta}$ for any $\delta\in(0,1)$ by an iteration argument, and we introduce some properties of the distance function. In section 3, we prove several lemmas concerning the linearized equation that will be used in the proof of our main result. Finally, in section 4, we establish the Step 1 of \eqref{e1.8} and complete  the proof of Theorem \ref{theorem1.1}.

\medskip
{\bf Notations.} Throughout this paper,We denote $\mathbf{P}_{k}$ the space of polynomials of order $k$ in $n$ variables: if $Q\in\mathbf{P}_k$, then
\[Q(x)=\sum_{\mu\in\mathbb{N}^n,|\mu|\leq k}q^{(\mu)}x^{(\mu)},\quad x\in\mathbb{R}^n.\]
We also use the standard notation $\lfloor t\rfloor$ for the integer part of a positive real number $t$.

\vspace{-0.1cm}

  \section{Preliminaries}
We begin this section by recalling some regularity properties of solutions to problem \eqref{e1.2}. Building on these results and the $C^{1,\delta_0}$ regularity of the free boundary, we then develop an iteration argument to improve this regularity to $C^{1,\delta}$ for any $\delta\in(0,1)$. Some useful properties of the distance function are also discussed.
 \begin{lemma}\label{lemma2.1}(\cite{DS2})
 Let $u$ be a minimizer for $\mathcal{E}_{\gamma}$ in $B_1$. Then
 \[\mathcal{H}^{n-1}(F(u)\cap B_{\frac{1}{2}})\leq C(n,\gamma)\]
 and the free boundary $F(u)$ is locally $C^{1,\delta_0}$ for some $\delta_0\in(0,1)$ except on a closed singular set $\Sigma_u\subset F(u)$ of Hausdorff dimension $n-3$.
 \end{lemma}

We consider the change of variables
\begin{equation}\label{e2.8}w:=c_0^{-\frac{1}{\alpha}}u^{\frac{1}{\alpha}},\end{equation}
with $c_0$ and $\alpha$ as in \eqref{e1.3}. Under this transformation, the problem \eqref{e1.2} can be viewed as a one-phase free boundary problem for $w$. A direct calculation shows that
\[\alpha w^{\alpha-1}\Delta w+\alpha(\alpha-1)w^{\alpha-2}|\nabla w|^2=-\alpha(1-\alpha)w^{-\alpha(\gamma+1)},\]
hence $w$ solves  a degenerate equation
\begin{equation}\label{e1.4}
\Delta w=(1-\alpha)\frac{|\nabla w|^2-1}{w}.
\end{equation}
Since $u$ is a viscosity solution of problem \eqref{e1.2}, the variable change implies that $w$ is a viscosity solution of \eqref{e1.4} in the following sense:
\begin{definition}\label{definition2.1}  We say $w: B_1\rightarrow\R^+$ satisfies \eqref{e1.4} in the viscosity sense, if
\begin{itemize}
\item [(i)] $w$ is $C^{\infty}$ and satisfies the equation in the classical sense in $\{w>0\}\cap B_1$;
\item [(ii)] if $x_0\in F(w):=\partial\{w>0\}\cap B_1$, then $w$ cannot touch $\psi\in\mathcal{C}^+$ (resp. $\mathcal{C}^-$) by above (resp. by below) at $x_0$ with
    \[\psi(x):=d(x)+\mu d(x)^{3-2\alpha},\]
with $\mu>0$ (resp. $\mu<0$).
\end{itemize}
\end{definition}

To prove the improvement of flatness, D. De Silva and O. Savin \cite{DS2} introduced the rescaling
\begin{equation}\label{e2.1}
\widetilde{w}=\frac{w-x_n}{\varepsilon}.
\end{equation}
They showed that this rescaled function is well approximated by a viscosity solution of the linearized equation
\begin{equation}\label{e2.2}
\begin{cases}
\Delta \phi+s\frac{\phi_n}{x_n}=0&\text{in\ }B_1^+:=B_1\cap\R^n_+\\
\frac{\partial\phi}{\partial x_n^{1-s}}=0 &\text{on\ }\{x_n=0\},
\end{cases}
\end{equation}
with $\phi_n:=\frac{\partial\phi}{\partial x_n}$ and $3-2\alpha=1-s>1$. Based on this approximation, they established the following improvement of flatness result:
\begin{lemma}\label{lemma2.2}(\cite{DS2})
Assume that $w$ is a viscosity solution of \eqref{e1.4} in $B_1$, and
\[(x_n-\varepsilon)^+\leq w\leq(x_n+\varepsilon)^+,\quad 0\in F(w),\]
for some $\varepsilon\leq\varepsilon_0(\gamma,n)$ small, universal. Then, there exists $r$ universal such that
\[(x\cdot\nu-\frac{\varepsilon}{2}r)^+\leq w\leq (x\cdot\nu+\frac{\varepsilon}{2}r)^+\quad\text{in }B_r,\]
for some unit direction $\nu$ with $|\nu-e_n|\leq C\varepsilon$.
\end{lemma}

An iteration of Lemma \ref{lemma2.2} (see Proposition 7.2 in \cite{DS2}) leads to the following result:
\begin{proposition}\label{proposition2.1}
Assume that $w$ is a viscosity solution of \eqref{e1.4} with $0\in F(w):=\partial\{w>0\}\cap B_1$. If for some small $\varepsilon>0$,
\[(x_n-\varepsilon)^+\leq w\leq(x_n+\varepsilon)^+,\quad\text{in }B_1\]
then, for some unit direction $\nu$ with $|\nu-e_n|\leq C\varepsilon$, we have
\[|w-(x\cdot\nu)^+|\leq C\varepsilon r^{1+\delta_0},\quad\text{in }B_r\]
for all small $r\in(0,1/2)$ and some $\delta_0\in(0,1)$.
\end{proposition}

The above improvement of flatness result gives a $C^{1,\delta_0}$ estimate with a fixed exponent $\delta_0$. In fact, this result can be improved to any $\delta\in(0,1)$ by a standard iteration argument. More precisely, we have the following result, which is of independent interest.
\begin{proposition}\label{proposition2.2}
Assume that $w$ is a viscosity solution of \eqref{e1.4} with $0\in F(w):=\partial\{w>0\}\cap B_1$. Given any $\delta\in(0,1)$, there exists a small $\varepsilon_0=\varepsilon_0(\delta,\gamma,d)>0$ such that if
\[(x_n-\varepsilon)^+\leq w\leq(x_n+\varepsilon)^+,\quad\text{in }B_1\]
for some $\varepsilon<\varepsilon_0$,
then, for some unit direction $\nu$ with $|\nu-e_n|\leq C\varepsilon$, we have
\[|w-(x\cdot\nu)^+|\leq C\varepsilon r^{1+\delta},\quad\text{in }B_r\]
for all small $r\in(0,1/2)$.
\end{proposition}
\begin{proof}
Step 1. We first claim that there exists $M=M(\gamma,n)>1$ and $r_0\in(0,1/2)$ such that
\[Mr_0<\frac{1}{2},\quad \log_{r_0}M>\delta-1,\]
and
\begin{equation}\label{e2.4}
(x\cdot\nu-M r_0^2\varepsilon)^+\leq w\leq(x\cdot\nu+M r_0^2\varepsilon)^+\quad\text{in }B_{r_0}
\end{equation}
for some unit vector $\nu$ with $|\nu-e_n|\leq C\varepsilon$.

Assume by contradiction that for $M$ and $r_0$ to be chosen, there exists a sequence $\varepsilon_k\rightarrow 0$ and a sequence of functions $w_k$ such that the hypothesis of the lemma holds, but the improvement is not achieved. Then, D. DE Silva and O. Savin showed that the normalized solutions
\[\widetilde{w}_k=\frac{w_k-x_n}{\varepsilon_k}\]
converge uniformly on compact sets to the graph of a H\"{o}lder limiting function $\widetilde{w}$ defined in $\overline{B}_1^+$ satisfying equation \eqref{e2.2}.
From Theorem 7.3 in \cite{DS2}, we know that
\[|\widetilde{w}(x)-\widetilde{w}(0)-a'\cdot x'|\leq C|x|^{1+\delta_0}\]
with $C>0$, $0<\delta_0<1$ depending only on $n,s$, and a vector $a'\in\mathbb{R}^{n-1}$. By a standard techniques in \cite{CC}, we obtain
\[|\Delta_{x'}\widetilde{w}|\leq C(\gamma,n)\quad\text{in }B_{\frac{1}{2}}\cap\{x_n>0\}.\]
For any fixed $x'\in\mathbb{R}^{n-1}$,   define
\[y(t):=\widetilde{w}(x',t).\]
Then $y$ satisfies the following  ODE equation:
 \begin{equation}\label{e2.3}
 y''+s\frac{y'}{t}=-\Delta_{x'}\widetilde{w}(x',t)\quad\text{for }t>0.
 \end{equation}
As a consequence, $y$ can be written as
\[y(t)=c_1 t^{1-s}+c_2+f(t),\]
where $f(t)$ is a function satisfies \eqref{e2.3} with
\[|f(t)|\leq Ct^2.\]
Combining the initial condition at $t=0$ with the "Neumann" condition $\frac{\partial \widetilde{w}}{\partial x_n^{1-s}}=0$, we obtain
\[c_1=0\quad\text{and}\quad c_2=\widetilde{w}(x',0).\]
Consequently, we have the expansion
\[|\widetilde{w}(x',x_n)-\widetilde{w}(0,0)-\nabla_{x'}\widetilde{w}(0,0)\cdot x'|\leq C(s,n)r^2\quad\text{in }B_r\cap\{x_n>0\}.\]
Rescaling back to the original $w_k$, we get
\[|w_k-x_n^+-\varepsilon_k\nabla_{x'}\widetilde{w}(0,0)\cdot x'|\leq C\varepsilon_k r^2+\varepsilon_k o(1)\quad\text{in }B_r.\]
We now choose constants  $M>0$ and $r_0>0$ such that
\[M=2C(s,n),\quad Mr_0<\frac{1}{2}\quad\text{and}\quad\log_{r_0}M>\delta_0-1.\]
For $k$ large such that $\varepsilon_k o(1)<\varepsilon_k C(s,n)r_0^2$, we obtain
\[|w_k-x_n^+-\varepsilon_k\nabla_{x'}\widetilde{w}(0,0)\cdot x'|\leq \varepsilon_kM r_0^2\quad\text{in }B_{r_0}.\]
This contradicts our assumption and completes the improvement of flatness.

\medskip
Step 2.  We complete the proof by iteration.

\medskip
Using the result from Step 1, an iteration argument gives
\[|w-(x\cdot\nu_k)^+|\leq C\varepsilon(Mr_0^2)^k\quad\text{in }B_{r_0^k}\]
with $|\nu_k-\nu_{k+1}|\leq C\varepsilon(Mr_0^2)^k$. In particular, for every $1\leq i<j$, we have
\[|\nu_i-\nu_j|\leq\sum_{k=i}^{j-1}|\nu_k-\nu_{k+1}|\leq\sum_{k=i}^{j-1}C\varepsilon(Mr_0^2)^k=C\varepsilon\frac{(Mr_0^2)^{i}}{1-Mr_0^2}.\]
Hence ${\nu_k}$ is a Cauchy sequence and converges to some unit vector $\nu_0$. Moreover,
\[ |\nu_i-\nu_0|\leq\sum_{k=i}^{\infty}|\nu_k-\nu_{k+1}|\leq C\varepsilon\frac{(Mr_0^2)^{i}}{1-Mr_0^2}.\]
Therefore,
\[\begin{split}
|w-(x\cdot\nu_0)^+|&\leq|w-(x\cdot\nu_i)^+|+|x\cdot(\nu_k-\nu_0)|\\
&\leq C\varepsilon(Mr_0^2)^i+C\varepsilon\frac{(Mr_0^2)^{i}}{1-Mr_0^2}=C\varepsilon(Mr_0^2)^i\left(1+\frac{1}{1-Mr_0^2}\right)\quad\text{in }B_r.
\end{split}\]
Let $r\in(0,1/2)$ be arbitrary and let $i\in\mathbb{N}$ be such that $r_0^{i+1}\leq r\leq r_0^{i}$, then we have
\[\begin{split}
|w-(x\cdot\nu_0)^+|&\leq C\varepsilon\left(1+\frac{1}{1-Mr_0^2}\right)(Mr_0)^ir_0^i\\
&\leq C\varepsilon\left(1+\frac{1}{1-Mr_0^2}\right)(Mr_0)^{\log_{r_0}r}r\\
&=C\varepsilon\left(1+\frac{1}{1-Mr_0^2}\right)r^{1+\log_{r_0}M}r\\
&\leq C\varepsilon\left(1+\frac{1}{1-Mr_0^2}\right)r^{1+\delta},
\end{split}\]
where we used $\log_{r_0}M>\delta-1$ in the last inequality. Then, we complete the proof.
\end{proof}

In our analysis we will need some properties of the distance function to the free boundary. It is well known that the distance function is $C^{1,\delta_0}$ near a $C^{1,\delta_0}$ boundary. In order to improve the free boundary regularity to $C^\infty$, we will work with regularized versions of the distance function.
\begin{lemma}\label{lemma2.3}(\cite{DS2})
Let $\partial\Omega\cap B_2$ be a $C^{k,\delta}$ graph for $\delta\in(0,1]$ and $k\in\mathbb{N}$. Then, there exists a generalized distance function $d(x)\in C^{\infty}(\Omega\cap B_1)\cap C^{k,\delta}(\overline{\Omega}\cap B_1)$ satisfying
\begin{itemize}
\item [(i)] \begin{equation}\label{e2.5}\frac{1}{C}\mathrm{dist}(x,\partial\Omega)\leq d(x)\leq C\mathrm{dist}(x,\partial\Omega),\quad\text{for }x\in\Omega\cap B_1.\end{equation}
\item [(ii)] \begin{equation}\label{e2.6}
d(x)=x_n+Q(x)+g(x),\quad\text{for } x\in\Omega\cap B_1,
\end{equation}
where $Q(x)\in \mathbf{P}_{\lfloor k+\delta\rfloor}$ is a polynomials of order $\lfloor k+\delta\rfloor$ with its zeroth and first order terms identically equal to zero and satisfying $\|Q\|_{L^{\infty}(B_1)}\leq C$, and where $g\in C^{\infty}(\Omega\cap B_1)$ satisfying $|\nabla g|\leq C|x|^{k-1+\delta}$;
\item [(iii)] For any $\lambda\in\mathbb{R}$, if $k=1$, then
\begin{equation}\label{e3.27}
\left|\Delta d^{\lambda}-\lambda(\lambda-1)d^{\lambda-2}\right|\leq Cd^{\delta+\lambda-2}\quad\text{in }\Omega\cap B_1,
\end{equation}
and if $k>2$, then
\begin{equation}\label{e3.28}
\Delta d^{\lambda}-\lambda(\lambda-1)d^{\lambda-2}= d^{\lambda-1}g(x)\quad\text{in }\Omega\cap B_1,
\end{equation}
with $g\in C^{k-2,\delta}(\overline{\Omega}\cap B_1)\cap C^{\infty}(\Omega\cap B_1)$.
\end{itemize}
\end{lemma}

We recall the barrier function introduced in \cite{DS2}. It serves as a generalized distance and satisfies additional properties that will be useful in our analysis.
\begin{lemma}\label{lemma2.4}(\cite{DS2})
Let $\rho\in(0,1/2)$ and let $\partial\Omega\cap B_1$ be a $C^{1,\delta}$ graph. Then, there exists a constant $C=C(\delta,\rho)$ and a barrier function $\phi\in C^{\infty}(\Omega\cap B_1)\cap C^{1,\delta}(\overline{\Omega}\cap B_r)$ for any $r\in(0,1)$ such that
\begin{itemize}
\item [(i)] \begin{equation}\label{e2.14}
\frac{1}{C}d(x)\leq\phi(x)\leq Cd(x),\quad\text{for any }x\in\Omega\cap B_{\rho};
\end{equation}
\item [(ii)] \begin{equation}\label{e2.15}
\phi(x)\geq \frac{1}{C},\quad\text{for any }x\in\Omega\cap B_{2\rho}^c;
\end{equation}
\item [(iii)] \begin{equation}\label{e2.16}
\phi(x)=x_n+g(x),\quad\text{for any }x\in\Omega\cap B_{1}
\end{equation}
where $g\in C^{\infty}(\Omega\cap B_2)$ with $|\nabla g|\leq C|x|^{\delta}$ in $\Omega\cap B_1$;
\end{itemize}
\item [(iv)] for any $\lambda\in\mathbb{R}$,
\begin{equation}\label{e2.17}
|\Delta\phi^{\lambda}-\lambda(\lambda-1)\phi^{\lambda-2}|\leq C|x|^{\delta+\lambda-2},\quad\text{for any }x\in\Omega\cap B_{1}.
\end{equation}
\end{lemma}

We are now in a position to obtain some properties of solutions to \eqref{e1.2} near regular points.
\begin{proposition}\label{proposition2.3}
Assume that $0\in F(u)$ be a regular point on the free boundary and  $u$ is a viscosity solution to problem \eqref{e1.2}. Then, there exists $r_0>0$ such that $\partial\Omega$ is $C^{1,\delta}$ for any $\delta\in(0,1)$. Moreover, there exists $C=C(\delta,\gamma,n)$ and $\rho>0$ such that
\begin{equation}\label{e2.7}
\left|u(x)-c_0 d^{\frac{2}{2+\gamma}}(x)\right|\leq C|x|^{\frac{2}{2+\gamma}+\delta}\quad\text{for }x\in\Omega\cap B_{\rho},
\end{equation}
with $c_0$ as in \eqref{e2.3}.
\end{proposition}

\begin{proof}
Let $w:=c_0^{-\frac{1}{\alpha}}u^{\frac{1}{\alpha}}$ as in \eqref{e2.8}, and extend $w$ as zero outside $\Omega$. Since $0\in F(u)$ is a regular point and $\nu(0)=-e_n$, there exist constants $\varepsilon_0,\rho_0>0$ such that
\begin{equation}\label{e2.9}
(x_n-\rho_{0}\varepsilon)\leq w(x)\leq x_n+\rho_{0}\varepsilon\quad\text{in }B_{1}
\end{equation}
for some $\varepsilon<\varepsilon_0$. Then, by Proposition \ref{proposition2.2}, there exists some unit vector $\nu$ and a universal constant $C>0$ such that
\[|w(x)-(x\cdot\nu)^+|\leq C\varepsilon r^{1+\delta}\quad\text{in }B_{r}\]
for all $\delta\in(0,1)$ and $r\in(0,\rho_0/2)$.
Therefore, we have
\[|w_r(x)-(x\cdot\nu)^+|\leq C\varepsilon r^{\delta}\quad\text{in }B_{1},\]
where $w_r(x)=\frac{w(rx)}{r}$. Moreover, since $0$ is a regular point and $\nu(0)=-e_n$, we have
\[w_r(x)\rightarrow(x_n)^+\quad\text{as }r\rightarrow 0^+,\]
which implies that $\nu=e_n$ and thus for any $r\in(0,\rho_0/2)$, we have
\[|w_r(x)-(x_n)^+|\leq C\varepsilon r^{\delta}\quad\text{in }B_{1}.\]
Rescaling back to $w$, we easily get
\begin{equation}\label{e2.10}|w(x)-(x_n)^+|\leq C|x|^{1+\delta}\quad\text{in }B_{\rho_0/2}.\end{equation}

Repeating the same argument, we can find $\rho_1$ such that for any $z\in B_{\rho_1}\cap\partial\Omega$,
\begin{equation}\label{e2.11}
|w(x+z)-\left(x\cdot(-\nu(z))\right)^+|\leq C|x|^{1+\delta}\quad\text{in }B_{\rho_1}(z),
\end{equation}
where $\nu(z)$ is the outer normal of $\partial\Omega$ at $z$. By plugging $z$ in \eqref{e2.10}, we have $|z\cdot e_n|=|(z_n)^+|\leq C|z|^{1+\delta}$ for any $\delta\in(0,1)$. Therefore, combining with \eqref{e2.10} and \eqref{e2.11}, we deduce that
\begin{equation}\label{e2.12}
\begin{split}
&|(x-z)\cdot e_n-((x-z)\cdot(-\nu(z)))^+|\\
\leq&|(x_n)^+-w(x)|+|w(x)-((x-z)\cdot(-\nu(z)))^+|+|(z_n)^+|\\
\leq &C\left(|x|^{1+\delta}+|z-x|^{1+\delta}+|z|^{1+\delta}\right).
\end{split}
\end{equation}

On the other hand, the flatness property \eqref{e2.9} implies that
\begin{equation}\label{e2.13}
\{x\in B_{\rho_0/2}| x_n\geq \varepsilon \rho_0 \}\subset\{w>0\}\quad\text{and}\quad |z_n|\leq\varepsilon\rho_0.
\end{equation}
Thus, by taking $\varepsilon_0$ and $\rho_0$ small enough, we can find $\{x\in B_{\rho_0/2}| x_n\geq \varepsilon \rho_0 \}$ satisfying $|x|\leq C_1|z|\leq C_2|x|$ such that
\[y=\frac{x-z}{|x-z|}\quad\text{with}\quad y_n\geq\frac{1}{4}.\]
Since $|x|$ and $|z|$ are comparable, in virtue of \eqref{e2.13}, we have $|x|\leq C_1|z-x|\leq C_2|x|$. Moreover, the $C^{1,\delta_0}$ regularity of $\partial\Omega$ yields that $(x-z)\cdot(-\nu(z))>0$. Therefore, from \eqref{e2.12}, we deduce that
\[\left|\frac{x-z}{|x-z|}\cdot(e_n+\nu(z))\right|\leq C|z|^{\delta},\]
which implies that
\[|\nu(0)-\nu(z)|=|e_n+\nu(z)|\leq C|z|^{\alpha}.\]
Now, we obtain that $\partial\Omega$ is $C^{1,\delta}$ in $B_{\rho_0/2}$.

The $C^{1,\delta}$ regularity of $\partial\Omega\cap B_{\rho_1}$ indicates the $C^{1,\delta}$ regularity of the distance function $d(x):=\mathrm{dist}(x,\partial\Omega)$. Indeed, since $\nabla d(0)=e_n$, we have $|d(x)-(x_n)^+|\leq C|x|^{1+\delta}$ for any $x\in B_{\rho_0/2}$. Combining this with \eqref{e2.10}, we deduce that
\[|w(x)-d(x)|\leq C|x|^{1+\delta}\quad\text{in }B_{\rho_0/2}\cap\Omega.\]
Applying the mean value theorem to the function $f(t)=t^{\frac{2}{2+\gamma}}$ and recall the definition of $w$ as in \eqref{e2.8}, we deduce that
\[\begin{split}
\left|u(x)-c_0d^{\frac{2}{2+\gamma}}(x)\right|=&c_0\left|w^{\frac{2}{2+\gamma}}(x)-c_0d^{\frac{2}{2+\gamma}}(x)\right|\\
=&c_0|w(x)-d(x)|\int_0^1\frac{2}{2+\gamma}\left(tw(x)+(1-t)d(x)\right)^{-\frac{\gamma}{2+\gamma}}dt\\
\leq& C|x|^{\frac{2}{2+\gamma}+\delta}
\end{split}\]
for any $x\in\Omega\cap B_{\rho_0/2}$. Then, we complete the proof.
\end{proof}

\vspace{-0.1cm}

\section{The linearized equation}
This section is devoted to the linearized version of problem \eqref{e1.2}. If $u$ is a viscosity solution, then the partial derivatives $u_i$ satisfy the following linearized equation
\[\Delta u_i=(\gamma+1)u^{-\gamma-2}u_i\quad\text{in }\{u>0\}\cap B_1.\]
Near regular points, the solution behaves like
\[u\sim c_0 d^{\frac{2}{2+\gamma}}.\]
Consequently, the linearized equation near the regular points becomes
\[\begin{split}
\Delta u_i=&(\gamma+1)\left[c_0 d^{\frac{2}{2+\gamma}}+o\left(d^{2-\frac{2}{2+\gamma}}\right)\right]^{-\gamma-2}u_i\\
=&(\gamma+1)c_0^{-\gamma-2}\frac{u_i}{d^2}+(\gamma+1)(-\gamma-2)c_0^{-\gamma-3}o\left(d^{\frac{2\gamma}{2+\gamma}}\right)\frac{u_i}{d^2}
\end{split}\]
This can be written as
\[\Delta u_i=\kappa\frac{u_i}{d^2}+fd^{-2-\frac{\gamma}{2+\gamma}},\]
where $f\in C^{\frac{2}{2+\gamma}}$, $f=0$ on $F(u)$ and
\begin{equation}\label{e3.1}
\kappa=(\gamma+1)c_0^{-\gamma-2}=\frac{2\gamma(\gamma+1)}{(2+\gamma)^2}.
\end{equation}

We now introduce the linearized operator
\begin{equation}\label{e3.2}
L_{\kappa}(v):=-\Delta v+\kappa\frac{v}{d^2}.
\end{equation}
Its one-dimensional characteristic exponents are given by:
\[L_{\kappa}(x^{a})=0\text{ for }x>0\quad\Leftrightarrow\quad a=\frac{1\pm\sqrt{1+4\kappa}}{2}.\]
Denoting
\begin{equation}\label{e3.3}
a_+=\frac{1+\sqrt{1+4\kappa}}{2}=1+\frac{\gamma}{2+\gamma}\quad\text{and}\quad a_-=\frac{1-\sqrt{1+4\kappa}}{2}=-\frac{\gamma}{2+\gamma},
\end{equation}
we observe that $a_+\in(1,3/2)$, $a_-\in(-1/2,0)$ and $\kappa=a_+(a_+-1)=a_-(a_--1)$.

\medskip
A key ingredient in our analysis is the study of one-dimensional solutions to the equation $L_{\kappa}(v)=f$. These allow us to understand the local behavior of solutions near the free boundary. This leads to the following lemma:
\begin{lemma}\label{lemma3.1}
Let $\kappa$ and $a_{\pm}$ be defined as in \eqref{e3.1} and \eqref{e3.3}, respectively. Given  $n_1,n_2\in\mathbb{N}\cup\{0\}$ and let $y:\mathbb{R}\rightarrow\mathbb{R}$ be a solution of
\begin{equation}\label{e3.4}
y''+\frac{\kappa}{x^2}y=x^{a_++n_1-1}+x^{a_-+n_2-1}
\end{equation}
satisfying
\begin{equation}\label{e3.5}
\lim_{x\rightarrow 0^+}\frac{y}{x^{a_-}}=0.
\end{equation}
Then,  the solution
\begin{equation}\label{e3.6}
y(x)=c_1 x^{a_+}+c_3 x^{a_++n_1+1}+c_4x^{a_-+n_2+1},
\end{equation}
where $c_i\in\mathbb{R}$ are constants.
\end{lemma}
\begin{proof}
The homogeneous ODE  associated with \eqref{e3.4} is a standard Cauchy-Euler equation whose solutions are of the form
\begin{equation}\label{e3.7}
y_0(x)=c_1 x^{a_+}+c_2 x^{a_-},\end{equation}
with $a_{\pm}$ defined as in \eqref{e3.3}. Condition \eqref{e3.5} then implies that $c_2=0$.
Now assume a particular solution of the form
\[y_p=c_3 x^{a_++n_1+1}+c_4 x^{a_-+n_2+1},\]
 then
\[L_{\kappa}(y_p)=c_3\left[-(a_++n_1+1)(a_++n_1)+\kappa\right]x^{a_++n_1-1}+c_4\left[-(a_-+n_2+1)(a_-+n_2)+\kappa\right]x^{a_-+n_2-1}.\]
Using  $\kappa=a_+(a_+-1)=a_-(a_--1)$, we derive
\[-(a_++n_1+1)(a_++n_1)+\kappa=(n_1+1)(-2a_+-n_1)<-2\]
and
\[-(a_-+n_2+1)(a_-+n_2)+\kappa=(n_2+1)(-2a_--1)<0.\]
Therefore,  $y_p$  indeed satisfies \eqref{e3.4} with
\[c_3=-\frac{1}{(n_1+1)(2a_++n_1)} \quad\text{and}\quad c_4=-\frac{1}{(n_2+1)(2a_--1)},\]
and the solution of \eqref{e3.4} is $y(x)=y_0(x)+y_p$, which is of the form \eqref{e3.6}.
\end{proof}
\medskip

The condition \eqref{e3.5} behave like a Dirichlet boundary condition on $\partial\Omega$, which implies the boundedness of $v$ due to $a_-\in(-1/2,0)$. Using this lemma, we will derive boundary estimates that provide the compactness needed for our blow-up arguments. To this end, we first establish a comparison result for the operator $L_{\kappa}$. The proof relies on the construction of explicit barriers, and the result can be stated as follows.
\begin{lemma}\label{lemma3.3}
Let $\Omega\cap B_1$ be a $C^{1,\delta}$ domain with $\delta\in(0,1)$ and let $v$ is a viscosity solution of
\begin{equation}\label{e3.18}
-\Delta v+\kappa\frac{v}{d^2}\leq 0\quad\text{in }\Omega\cap B_1,
\end{equation}
satisfying
\begin{equation}\label{e3.19}
\lim_{x\rightarrow\partial\Omega}\frac{v}{d^{a_-}}\leq 0.
\end{equation}
Then, there exists $\rho\in(0,1)$ such that for any $r\in(0,\rho)$, if $v\leq0$ on $\partial\Omega\cap B_r$, then $v\leq0$ in $\Omega\cap B_r$.
\end{lemma}
\begin{proof}
Since $|\nabla d(x)|=1$ on $\partial\Omega\cap B_1$, the $C^{1,\delta}$ regularity of $\Omega$ implies that $|\nabla d|^2=1+O\left(d^{\delta}\right)$. By a direct calculation,
\[\begin{split}
\Delta\left(d^{\frac{1}{2}}|\ln d|^{\frac{1}{2}}\right)=&\Delta\left(d^{\frac{1}{2}}\right)|\ln d|^{\frac{1}{2}}+2\nabla\left(d^{\frac{1}{2}}\right)\cdot\nabla\left(|\ln d|^{\frac{1}{2}}\right)+d^{\frac{1}{2}}\Delta\left(|\ln d|^{\frac{1}{2}}\right)\\
=&-\frac{d^{\frac{1}{2}}|\ln d|^{\frac{1}{2}}}{4d^2}+|\ln d|^{\frac{1}{2}}O\left(d^{\delta-\frac{3}{2}}\right)-\frac{d^{\frac{1}{2}}|\nabla d|^2}{2|\ln d|^{\frac{1}{2}}d^2}\\
&+\frac{d^{\frac{1}{2}}}{2}\left[-\frac{|\nabla d|^2}{2|\ln d|^{\frac{3}{2}}d^2}+\frac{|\nabla d|^2-d\Delta d}{|\ln d|^{\frac{1}{2}}d^2}\right]\\
=&-\frac{d^{\frac{1}{2}}|\ln d|^{\frac{1}{2}}}{4d^2}-\frac{1}{4|\ln d|^{\frac{3}{2}}d^{\frac{3}{2}}}\left(1+O\left(|\ln d|^2 d^{\delta}\right)\right),
%-\frac{\Delta d}{2|\ln d|^{\frac{1}{2}}d^{\frac{1}{2}}}
\end{split}\]
which implies that
\[L_{\kappa}\left(d^{\frac{1}{2}}|\ln d|^{\frac{1}{2}}\right)=\left(\frac{1}{4}+\kappa\right)\frac{d^{\frac{1}{2}}|\ln d|^{\frac{1}{2}}}{d^2}+\frac{1}{4|\ln d|^{\frac{3}{2}}d^{\frac{3}{2}}}\left(1+O\left(|\ln d|^2 d^{\delta}\right)\right).\]
From \eqref{e3.27} and the identity $a_-(a_--1)=\kappa$, we deduce that
\[\left|L_{\kappa}\left(d^{a_-}\right)\right|\leq C d^{a_--2}\rho^{\delta}.\]
Let $\psi=d^{\frac{1}{2}}|\ln d|^{\frac{1}{2}}+\left(\frac{1}{4}+\kappa\right)d^{a_-}$, then we obtain
\begin{equation}\label{e3.20}
L_{\kappa}(\psi)\geq \left(\frac{1}{4}+\kappa\right)\frac{d^{a_-}}{d^2}\left[d^{\frac{1}{2}-a_-}|\ln d|^{\frac{1}{2}}-C\rho^{\delta}\right]+\frac{1}{4|\ln d|^{\frac{3}{2}}d^{\frac{3}{2}}}\left(1+O\left(|\ln d|^2 d^{\delta}\right)\right)>0\quad\text{in }B_{\rho}\cap\Omega
\end{equation}
by shrinking $\rho$ small enough. Hence, if for  some $r\leq \rho$, we have that $v\leq0$ on $\Omega\cap\partial B_r$, the boundary condition \eqref{e3.19} implies that
\begin{equation}\label{e3.21}
\lim_{x\rightarrow\partial(\Omega\cap B_r)}\frac{v(x)}{\psi(x)}\leq0.
\end{equation}
This means that either $v\leq0$ in $\Omega\cap B_r$ or $\frac{v}{\psi}$ attains a positive maximum in $\Omega\cap B_r$. If the later holds, then there exists $M>0$ and $z\in\Omega\cap B_r$ such that
\[v(z)=M\psi(z)\quad\text{and}\quad v\leq M\psi\text{  in }\Omega\cap B_r.\]
The maximum implies that $-\Delta v(z)\geq -M\Delta\psi(z)$, thus we have
\[0\geq L_{\kappa}v(z)\geq ML_{\kappa}\psi(z)>0,\]
which is a contradiction. Therefore, we conclude that
\[v\leq0\quad\text{in}\quad \Omega\cap B_r.\]
\end{proof}

With the comparison principle, we obtain boundary estimates that yield the compactness required for our blow-up analysis. These estimates allow us to extract convergent subsequences of rescaled solutions. A key tool in this convergence is the following lemma.
\begin{lemma}\label{lemma3.2}
Let $\Omega\subset\mathbb{R}^n$ be a $C^{k,\delta}$ domain for any $\delta\in(0,1)$ and $k\in\mathbb{N}$. Let $d(x)$ be the regularized distance to $\partial\Omega$ and let $v\in C(\Omega\cap B_1)$ be any solution to
\begin{equation}\label{e3.14}
-\Delta v+\kappa\frac{v}{d^2}=f\quad\text{in }\Omega\cap B_1,
\end{equation}
satisfying
\begin{equation}\label{e3.15}
\lim_{x\rightarrow\partial\Omega}\frac{v}{d^{a_-}}=0,
\end{equation}
where $f\in C(\Omega\cap B_1)$ satisfies
\begin{equation}\label{e3.8}
f d^{\frac{3}{2}-a_--\varepsilon}\in L^{\infty}(\Omega\cap B_1)
\end{equation}
for some $\varepsilon>0$. Then, there exists $\rho\in(0,1)$ such that
\begin{equation}\label{e3.17}
|v|\leq C\left(\|v\|_{L^{\infty}(\Omega\cap B_1)}+\|f\phi^{\frac{3}{2}-a_--\varepsilon}\|_{L^{\infty}(\Omega\cap B_1)}\right)d^{\frac{1}{2}+a_-}\quad\text{in }\Omega\cap B_{\rho},
\end{equation}
where $\phi$ is the barrier function constructed in Lemma \ref{lemma2.4}, and $\rho=\rho(\delta,\varepsilon,n)$.

Moreover, we have the following Lipschitz estimate
\begin{equation}\label{e3.9}
[v]_{C^{1/2+a_-}(\Omega\cap B_{\rho})}\leq C\left(\|v\|_{L^{\infty}(\Omega\cap B_1)}+\|f\phi^{\frac{3}{2}-a_--\varepsilon}\|_{L^{\infty}(\Omega\cap B_1)}\right).
\end{equation}
\end{lemma}

\begin{proof}
We only prove the case for $k=1$; for $k>1$ the proof is the same using  Lemma \ref{lemma2.3}.  Let $\rho\in(0,1/2)$ to be determined later and let $\phi$ be the barrier function constructed in Lemma \ref{lemma2.4}.
 %Using  the identity $a_-(a_--1)=\kappa$, we derive that
%\[\begin{split}
%L_{\kappa}\left(\phi^{\frac{1}{2}+a_-}\right)=&-\Delta\phi^{\frac{1}{2}+a_-}+\kappa\phi^{-\frac{3}{2}+a_-}\\
%=&-\Delta\phi^{\frac{1}{2}+a_-}+(\frac{1}{2}+a_-)(a_--\frac{1}{2})\phi^{-\frac{3}{2}+a_-}-(a_--\frac{1}{4})\phi^{-\frac{3}{2}+a_-}
%\end{split}\]
%in $B_{2\rho}\cap\Omega$.
Thanks to the property \eqref{e2.17} of the barrier function $\phi$, we can combine \eqref{e3.27} with  the identity $a_-(a_--1)=\kappa$ to  deduce that
\[\left|L_{\kappa}\left(\phi^{\frac{1}{2}+a_-}\right)+\left(a_--\frac{1}{4}\right)\phi^{-\frac{3}{2}+a_-}\right| \leq C\phi^{-\frac{3}{2}+a_-}\rho^{\delta} \quad\text{in }B_{2\rho}\cap\Omega,\]
which implies that
\begin{equation}\label{e3.10}
L_{\kappa}\left(\phi^{\frac{1}{2}+a_-}\right)\geq -C\phi^{-\frac{3}{2}+a_-}\rho^{\delta}-(a_--\frac{1}{4})\phi^{-\frac{3}{2}+a_-} \quad\text{in }B_{2\rho}\cap\Omega.
\end{equation}
By a similar calculation, since $(\frac{1}{2}+a_-+\varepsilon)(a_--\frac{1}{2}+\varepsilon)=\kappa+(\frac{1}{2}+\varepsilon)(2a_--\frac{1}{2}+\varepsilon)$, we have
\[\left|L_{\kappa}\left(\phi^{\frac{1}{2}+a_-+\varepsilon}\right)+\left(\frac{1}{2}+\varepsilon\right)
\left(2a_--\frac{1}{2}+\varepsilon\right)\phi^{-\frac{3}{2}+a_-+\varepsilon}\right|\leq C\phi^{-\frac{3}{2}+a_-+\varepsilon}\rho^{\delta} \quad\text{in }B_{2\rho}\cap\Omega\]
which indicates that
\begin{equation}\label{e3.11}
L_{\kappa}\left(\phi^{\frac{1}{2}+a_-+\varepsilon}\right)\leq C\phi^{-\frac{3}{2}+a_-+\varepsilon}\rho^{\delta}-\left(\frac{1}{2}+\varepsilon\right)
\left(2a_--\frac{1}{2}+\varepsilon\right)\phi^{-\frac{3}{2}+a_-+\varepsilon} \quad\text{in }B_{2\rho}\cap\Omega.
\end{equation}

Therefore, we can fix $\rho>0$ small enough such that for some $\xi>0$,
\begin{equation}\label{e3.12}
\begin{split}
&L_{\kappa}\left(\phi^{\frac{1}{2}+a_-+\varepsilon}\right)-L_{\kappa}\left(\phi^{\frac{1}{2}+a_-}\right)\\
\leq& C\phi^{-\frac{3}{2}+a_-+\varepsilon}\rho^{\delta}-\left(\frac{1}{2}+\varepsilon\right)
\left(2a_--\frac{1}{2}+\varepsilon\right)\phi^{-\frac{3}{2}+a_-+\varepsilon}\\
&+C\phi^{-\frac{3}{2}+a_-}\rho^{\delta}+(a_--\frac{1}{4})\phi^{-\frac{3}{2}+a_-}\\
=&\phi^{-\frac{3}{2}+a_-+\varepsilon}\left[C\rho^{\delta}-\left(\frac{1}{2}+\varepsilon\right)
\left(2a_--\frac{1}{2}+\varepsilon\right)+C\phi^{-\varepsilon}\rho^{\delta}+\left(a_--\frac{1}{4}\right)\phi^{-\varepsilon}\right]\\
\leq&-\xi \phi^{-\frac{3}{2}+a_-+\varepsilon}
\end{split}
\end{equation}
in $B_{2\rho}\cap\Omega$. Thus, taking $\phi$ to be subordinate to $\rho$, it follows from \eqref{e2.15} that
\begin{equation}\label{e3.13}
 \phi^{\frac{1}{2}+a_-}-\phi^{\frac{1}{2}+a_-+\varepsilon}\geq\xi\quad\text{on }\partial B_{2\rho}\cap\Omega.
\end{equation}

Dividing both sides of \eqref{e3.14} by
\[\Lambda:=\frac{\|v\|_{L^{\infty}(\Omega\cap B_1)}+\|f\phi^{\frac{3}{2}-a_--\varepsilon}\|_{L^{\infty}(\Omega\cap B_1)}}{\xi}.\]
Let $v_1:=\frac{v}{\Lambda}$, then $\|v_1\|_{L^{\infty}(\Omega\cap B_1)}\leq\xi$ and satisfies
\[L_{\kappa}(v_1)=f_1,\]
where $f_1=\frac{f}{\Lambda}\leq\xi\phi^{-\frac{3}{2}+a_-+\varepsilon}$ in $\Omega\cap B_1$, and \eqref{e3.15} also holds.

Let  $w=v_1+\phi^{\frac{1}{2}+a_-+\varepsilon}-\phi^{\frac{1}{2}+a_-}\leq0$ on $\Omega\cap\partial B_{2\rho}$. Therefore,
\[L_{\kappa}w=f_1-\xi\phi^{-\frac{3}{2}+a_-+\varepsilon}\leq0\quad\text{in }\Omega\cap B_1.\]
 Applying Lemma \ref{lemma3.3}, we get $w\leq0$ in $\Omega\cap B_{\rho}$, i.e.,
 \begin{equation}\label{e3.16}
 v_1\leq \phi^{\frac{1}{2}+a_-}-\phi^{\frac{1}{2}+a_-+\varepsilon}\leq C d^{\frac{1}{2}+a_-}\quad\text{in }\Omega\cap B_{\rho}.
 \end{equation}
 By a similar argument to $-v_1$, and combining it with \eqref{e3.16}, we derive \eqref{e3.17}.

\medskip
  We complete the proof by deducing \eqref{e3.9}.
 For any $x_0\in\Omega$, take any $B_r(x_0)\subset B_{2r}(x_0)\subset\Omega$ so that $d(x_0,\partial\Omega)\leq Cr$. Define $v_r(x)=v(x_0+2rx)$, then it solves
 \[-\Delta v_r=r^2\left(f-\kappa\frac{v_r}{d^2}\right)\quad\text{in } B_1.\]
Let $[u]_{C^{\tau}(\Omega)}$ denote the H\"{o}lder seminorm of $u$:
 \[[u]_{C^{\tau}(\Omega)}:=\sup_{x,y\in\Omega,x\neq y}\frac{|u(x)-u(y)|}{|x-y|^{\tau}}.\]
By the interior regularity estimates in \cite{K}, we obtain
 \[\begin{split}
 [v_r]_{C^{1}(B_{1})}\leq& C\left(\|r^2\left(f-\kappa\frac{v_r}{d^2}\right)\|_{L^{\infty}(B_{2r}(x_0))}+\|v_r\|_{L^{\infty}(B_r(x_0))}\right)\\
 \leq& Cr^{\frac{1}{2}+a_-}\left(\|v\|_{L^{\infty}(\Omega\cap B_1)}+\|f\phi^{\frac{3}{2}-a_--\varepsilon}\|_{L^{\infty}(\Omega\cap B_1)}\right),
 \end{split}\]
 where we applied \eqref{e3.17} in the last inequality.
Translating it back to $v$, we get
 \[[v]_{C^{1/2+a_-}(B_{1})}\leq C \left(\|v\|_{L^{\infty}(\Omega\cap B_1)}+\|f\phi^{\frac{3}{2}-a_--\varepsilon}\|_{L^{\infty}(\Omega\cap B_1)}\right).\]
Therefore, the estimate  \eqref{e3.9} follows from the previous inequality combined with a covering argument.

\end{proof}

\begin{lemma}\label{lemma3.4}
Let $k\in\mathbb{N}$, $\delta\in(0,1)$ and let $\Omega\subset\mathbb{R}^n$ be a $C^{k,\delta}$ domain. Let $d$ be the regularized distance function to $\partial\Omega$ and let $v\in C(\Omega\cap B_1)$ be a solution of
\begin{equation}\label{e3.23}
L_{\kappa}(v)(x)=\left(Q_1(x)+g_1(x)\right)\left(\frac{d(rx)}{r}\right)^{a_--1}+\left(Q_2(x)+g_2(x)\right)\left(\frac{d(rx)}{r}\right)^{a_+-1}\quad\text{in }\frac{1}{r}\Omega\cap B_1,
\end{equation}
with $Q_i\in \mathbf{P}_{l}$ for some $l\in\mathbb{N}$ and $g_i\in L^{\infty}(\frac{1}{r}\Omega\cap B_1)$ for $i=1,2$. Then, there exists $\rho_0>0$ and $C=C(l,\beta)$ such that
\begin{equation}\label{e3.22}
\|Q_i\|_{L^{\infty}(\frac{1}{r}\Omega\cap B_1)}\leq C\left(\|g_1\|_{L^{\infty}(\frac{1}{r}\Omega\cap B_1)}+\|g_2\|_{L^{\infty}(\frac{1}{r}\Omega\cap B_1)}+\|v\|_{L^{\infty}(\frac{1}{r}\Omega\cap B_1)}\right),
\end{equation}
for any $r\in(0,\rho_0)$ and $i=1,2$.
\end{lemma}
\begin{proof}
Assume by contradiction that there exists sequences of domains $\Omega_{k}:=\frac{1}{r_k}\Omega$ with $\|\partial\Omega_k\cap B_2\|_{C^{k,\delta}}\leq 1$, a sequence of $\{r_k\}$ such that $r_k\rightarrow 0$ as $k\rightarrow\infty$, and a sequence of polynomials $Q_i^{k}$, functions $g_i^k\in L^{\infty}(\frac{1}{r_k}\Omega\cap B_1)$ and solutions $v_k$ satisfies \eqref{e3.23} but \eqref{e3.22} does not hold.
Let  $w_k:=\frac{v_k}{\Lambda_k}$ with
\[\Lambda_k:=\|Q_1^{k}\|_{L^{\infty}(\Omega_k\cap B_1)}+\|Q_2^{k}\|_{L^{\infty}(\Omega_k\cap B_1)},\]
then we obtain
\begin{equation}\label{e3.25}
L_{\kappa}(w_k)(x)=\left(\widetilde{Q}_1^k(x)+\widetilde{g}_1^k(x)\right)\left(\frac{d(r_kx)}{r_k}\right)^{a_--1}+\left(\widetilde{Q}_2^k(x)+\widetilde{g}_2^k(x)\right)\left(\frac{d(r_kx)}{r_k}\right)^{a_+-1}\quad\text{in }\frac{1}{r_k}\Omega\cap B_1,
\end{equation}
where $\widetilde{Q}_i^k:=\frac{Q_i^k}{\Lambda_k}$ and $\widetilde{g}_i^k=\frac{g_i^k}{\Lambda_k}\rightarrow 0$ as $k\rightarrow\infty$ for $i=1,2$ and by Lemma \ref{lemma3.2}, we have
%Multiplying the right hand side of \eqref{e3.25} by $d^{\frac{3}{2}-a_--\varepsilon}$  and applying \eqref{e3.17} of Lemma \ref{lemma3.2}, we obtain
\begin{equation}\label{e3.26}
\limsup_{k\rightarrow\infty}\|w_k\|_{L^{\infty}(\Omega_k\cap B_1)}=0.
\end{equation}
Notice that $\|\widetilde{Q}_1^{k}\|_{L^{\infty}(\Omega_k\cap B_1)}+\|\widetilde{Q}_2^{k}\|_{L^{\infty}(\Omega_k\cap B_1)}=1$, then, up to a subsequence, $\widetilde{Q}_i^{k}\rightarrow P_i$   with at least one of $P_i\neq0$. From \eqref{e2.6}, we have that, up to a subsequence,
\[\frac{d(r_kx)}{r_k}\rightarrow x_n\]
by extending it by zero outside $\Omega_k\cap B_1$. Therefore, by the stability of viscosity solutions under locally uniform converge, we deduce from \eqref{e3.25} and \eqref{e3.26} that
\[0=P_1 x_n^{a_+-1}+P_2 x_n^{a_--1},\]
which is a contradiction with the fact that at lest one of $P_i$ is not zero and $a_+-a_-\notin\mathbb{N}$.

\end{proof}

\begin{lemma}\label{lemma3.5}
Let $l\in\mathbb{N}$ and $Q_i\in \mathbf{P}_{l}$ for $i=1,2$. If $u$ satisfies
\begin{equation}\label{e3.23}
\begin{cases}
L_{\kappa}u=Q_1(x)x_n^{a_+-1}+Q_2(x)x_n^{a_--1}&\text{in}\quad\{x_n>0\}\\
\|u\|_{L^{\infty}(B_R)}\leq CR^{\delta+k+a_-}&\text{for}\quad k\in\mathbb{N}
\end{cases}
\end{equation}
with the boundary condition
\begin{equation}\label{e3.24}
\lim_{x_n\rightarrow 0^+}\frac{u(x',x_n)}{x_n^{a_-}}=0.
\end{equation}
Then, $u=P_1(x)x_n^{a_+}+P_2(x)x_n^{a_-+1}$ with $P_1\in\mathbf{P}_{\lfloor k+\delta-a_++a_-\rfloor-1}$ and  $P_2\in\mathbf{P}_{\lfloor k+\delta\rfloor-1}$.
\end{lemma}
\begin{proof}
Given any $R>1$, define $u_{R}(x)=R^{-k-\delta-a_-}u(Rx)$, then
\[\|u_{R}(x)\|_{L^{\infty}(B_1)}\leq R^{-k-\delta-a_-}\|u\|_{L^{\infty}(B_R)}\leq C.\]
Notice that,
\begin{equation}\label{e3.29}\begin{split}
(L_{\kappa}u_{R})(x)=&R^{-k-\delta-a_-+2}(L_{\kappa}u)(Rx)\\
=&R^{-k-\delta+a_+-a_-+1}Q_1(Rx)x_n^{a_+-1}+R^{-k-\delta+1}Q_2(Rx)x_n^{a_--1}
\end{split}\end{equation}
Using Lemma \ref{lemma3.4} in the flat domain $\Omega=\mathbb{R}^n_+$, we deduce that
\[R^{-k-\delta+a_+-a_-+1}\|Q_1(Rx)\|_{L^{\infty}(B_1)}+R^{-k-\delta+1}\|Q_2(Rx)\|_{L^{\infty}(B_1)}\leq C,\]
which implies that
\[Q_1\in\mathbf{P}_{\lfloor k+\delta-a_++a_-\rfloor-1}\quad\text{and}\quad Q_2\in\mathbf{P}_{\lfloor k+\delta\rfloor-1}.\]
On the other hand, multiplying the right hand side of \eqref{e3.29} by $d^{\frac{3}{2}-a_--\varepsilon}$, we deduce from Lemma \ref{lemma3.2} that
\[[u_{R}]_{C^{1/2+a_-}(B_{1/2})}\leq C\left(1+\|Q_1\|_{L^{\infty}(B_1)}+\|Q_2\|_{L^{\infty}(B_1)}\right).\]
Therefore,
\begin{equation}\label{e3.30}
[u]_{C^{1/2+a_-}(B_{R/2})}=R^{k+\delta-\frac{1}{2}}[u_{R}]_{\frac{1}{2}+a_-;B_{1/2}}\leq C R^{k+\delta-\frac{1}{2}}.
\end{equation}
We consider the quotients
\begin{equation}\label{e3.31}
u_1(x)=\frac{u(x+t\tau)-u(x)}{t^{\frac{1}{2}+a_-}}
\end{equation}
where $t\in(0,\frac{R}{2})$ and $\tau\in\mathbb{S}^{n-1}$ with $\tau_n=0$.
Since $\tau$ is orthogonal to $e_n$, we have that $u_1$ satisfies the boundary condition \eqref{e3.24}, together with \eqref{e3.30}, we deduce that
\begin{equation}\label{e3.32}
\begin{cases}
L_{\kappa}u_1=\widetilde{Q}_1(x)x_n^{a_+-1}+\widetilde{Q}_2(x)x_n^{a_--1}&\text{in }\{x_n>0\}\\
\|u_1\|_{L^{\infty}(B_R)}\leq C R^{k+\delta-\frac{1}{2}},&
\end{cases}
\end{equation}
where $\widetilde{Q}_i(x)=\frac{Q_i(x+t\tau)-Q_i(x)}{t^{\frac{1}{2}+a_-}}$. By a similar argument, we obtain
\[[u_1]_{C^{1/2+a_-}(B_{R/2})}\leq C R^{k+\delta-a_--1}.\]
Repeating this argument inductively $m$ times until $(m+1)(\frac{1}{2}+a_-)>k+\delta+a_-$, we deduce that
\[[u_m]_{C^{1/2+a_-}(B_{R/2})}\leq C R^{k+\delta+a_--(m+1)(\frac{1}{2}+a_-)}\rightarrow 0\quad\text{as }R\rightarrow\infty.\]
The boundary condition implies that $u_m=0$. Iterating backwards the difference quotients (see \cite{ARO}), we conclude that
\begin{equation}\label{e3.33}
u(x',x_n)=\sum_{|\mu|\leq M} C_{\mu}(x_n)(x')^{\mu},
\end{equation}
with $M=\lfloor k+\delta+a_-\rfloor$.

 Next, we prove by induction that
\[C_{\mu}(x_n)=P_{1,\mu}(x_n)x_n^{a_+}+P_{2,\mu}(x_n)x_n^{a_-+1}\]
for polynomials $P_{1,\mu}\in\mathbf{P}_{\lfloor k+\delta-a_++a_-\rfloor-1}$ and $P_{2,\mu}\in\mathbf{P}_{\lfloor k+\delta\rfloor-1}$.

For the base case $|\mu_0|=M$,
\[\partial_{x'}^{\mu_0}u(x',x_n)=\sum_{|\mu|\leq\mu_0}C_{\mu}(x_n)\partial_{x'}^{\mu_0}(x')^{\mu}=\mu_0! C_{\mu_0}(x_n).\]
Therefore, taking tangential derivatives of maximal order $\mu_0$ on both sides of the first equation for \eqref{e3.23}, we obtain that
\[L_{\kappa,1}(C_{\mu_0})=0,\]
where $L_{\kappa,1}(y)=-y''+\frac{\kappa}{x_n^2}y$. Combining with Lemma \ref{lemma3.1} and the boundary condition, we deduce that
\[C_{\mu_0}(x_n)=C_{1,\mu_0}x_n^{a_+}.\]

Assume now the result holds for $|\mu|>i$, we only need to show that the result holds for $\mu$ with $|\mu|=i$. By the induction hypothesis,
\[\begin{split}
\partial_{x'}^{i}u=&\sum_{j\leq|\mu|\leq M}C_{\mu}(x_n)\partial^{i}(x')^{\mu}\\
=&i! C_{i}(x_n)+\sum_{i+1\leq|\mu|\leq M}(x')^{\mu-i}\left(P_{1,\mu}(x_n)(x_n)^{a_+}+P_{2,\mu}(x_n)(x_n)^{a_-+1}\right).
\end{split}\]
Moreover,
\[\begin{split}
&L_{\kappa}\left[(x')^{\mu-i}\left(P_{1,\mu}(x_n)(x_n)^{a_+}+P_{2,\mu}(x_n)(x_n)^{a_-+1}\right)\right]\\
=&-\Delta\left[(x')^{\mu-i}\left(P_{1,\mu}(x_n)(x_n)^{a_+}+P_{2,\mu}(x_n)(x_n)^{a_-+1}\right)\right]+\kappa\frac{(x')^{\mu-i}\left(P_{1,\mu}(x_n)(x_n)^{a_+}+P_{2,\mu}(x_n)(x_n)^{a_-+1}\right)}{x_n^2}\\
=&-\Delta\left((x')^{\mu-i}\right)\left[P_{1,\mu}(x_n)(x_n)^{a_+}+P_{2,\mu}(x_n)(x_n)^{a_-+1}\right]-(x')^{\mu-i}\left[P''_{1,\mu}(x_n)(x_n)^{a_+}+2a_+P_{1,\mu}'(x_n)(x_n)^{a_+-1} \right.\\
&\left.+P''_{2,\mu}(x_n)(x_n)^{a_-+1}+2(a_-+1)P'_{1,\mu}(x_n)(x_n)^{a_-}-2a_-P_{2,\mu}(x_n)(x_n)^{a_--1}\right].
\end{split}\]
Therefore, if we take tangential derivatives of order $\mu$ with $|\mu|=i$, we obtain
\[\begin{split}
i!L_{\kappa,1}\left(C_{i}(x_n)\right)=&\partial_{x'}^i Q_1(x)x_n^{a_+-1}+\partial_{x'}^i Q_2(x)x_n^{a_--1}\\
&-\sum_{i+1\leq|\mu|\leq M}\Delta\left((x')^{\mu-i}\right)\left[P_{1,\mu}(x_n)(x_n)^{a_+}+P_{2,\mu}(x_n)(x_n)^{a_-+1}\right]\\
&-\sum_{i+1\leq|\mu|\leq M}(x')^{\mu-i}\left[P''_{1,\mu}(x_n)(x_n)^{a_+}+2a_+P_{1,\mu}'(x_n)(x_n)^{a_+-1} \right.\\
&\left.+P''_{2,\mu}(x_n)(x_n)^{a_-+1}+2(a_-+1)P'_{1,\mu}(x_n)(x_n)^{a_-}-2a_-P_{2,\mu}(x_n)(x_n)^{a_--1}\right].
\end{split}\]
Since $L_{\kappa,1}\left(C_{i}(x_n)\right)$ does not depend on $x'$, we derive that
\[L_{\kappa,1}\left(C_{i}(x_n)\right)=\widetilde{P}_1(x_n)x_n^{a_+-1}+\widetilde{P}_2(x_n)x_n^{a_--1}\]
for some polynomials $\widetilde{P}_1, \widetilde{P}_2$. Together with Lemma \ref{lemma3.1} and the boundary condition \eqref{e3.24}, we conclude that \[C_{i}(x_n)=P_{1,i}(x_n)x_n^{a_+}+P_{2,i}(x_n)x_n^{a_-+1}.\]
\end{proof}

We now state a technical lemma that controls the polynomials arising in the least squares approximation of our blow-up arguments.
\begin{lemma}\label{lemma3.6}
Let $k\in\mathbb{N}$, $\delta\in(0,1)$ and let $\Omega$ be such that $\|\partial\Omega\cap B_2\|_{C^{k,\delta}}\leq 1$. Assume that for some $\rho>0$ and any $r\in(0,\rho)$, the function $v\in L^{\infty}(B_2)$ satisfies
\begin{equation}\label{e3.35}
\|v-Q_{1,s}d^{a_+}-Q_{2,s}d^{a_-+1}\|_{L^{\infty}(B_s)}\leq \zeta(s)s^{k+\delta+a_-}
\end{equation}
for some polynomials  $Q_{i,r}:=\sum_{|\mu|\leq k_i}q_{i,r}^{(\mu)}x^{\mu}$ with $k_1=\lfloor k+\delta-a_++a_-\rfloor-1$ and $k_2=\lfloor k+\delta\rfloor-1$, where $\zeta(s):(0,\rho)\rightarrow (0,\infty)$ is non-increasing function with
\begin{equation}\label{e3.34}
\lim_{r\rightarrow 0^+}\zeta(r)=+\infty.
\end{equation}
Then, for any $r\in(0,\rho)$, the function
\begin{equation}\label{e3.37}
w_r(x):=\frac{v(rx)-Q_{1,r}(rx)d^{a_+}(rx)-Q_{2,r}(rx)d^{a_-+1}(rx)}{\zeta(r)r^{^{k+\delta+a_-}}}
\end{equation}
satisfy the  uniform bound
\begin{equation}\label{e3.38}
\|w_r\|_{L^{\infty}(B_R)}\leq CR^{k+\delta+a_-}
\end{equation}
for $R>1$.
\end{lemma}
\begin{proof}
Applying the triangular inequality and the monotonicity of $\zeta(r)$, we have
\[\begin{split}
&\|(Q_{1,2r}-Q_{1,r})d^{a_+}+(Q_{2,2r}-Q_{2,r})d^{a_-+1}\|_{L^{\infty}(B_r)}\\
\leq&\|v-Q_{1,r}d^{a_+}-Q_{2,r}d^{a_-+1}\|+\|v-Q_{1,2r}d^{a_+}-Q_{2,2r}d^{a_-+1}\|_{L^{\infty}(B_{2r})}\\
\leq&\zeta(r)r^{k+\delta+a_-}+\zeta(2r)(2r)^{k+\delta+a_-}\\
\leq& C\zeta(r)r^{k+\delta+a_-}.
\end{split}\]
Notice that, by the equivalence of norms in finite dimensional  spaces, there exists a constant $C$ such that
\begin{equation}\label{e3.39}\begin{split}
&r^{a_+}\|Q_{1,2r}-Q_{1,r}\|_{L^{\infty}(B_{r})}+r^{a_-+1}\|Q_{2,2r}-Q_{2,r}\|_{L^{\infty}(B_{r})}\\
=&r^{a_+}\|(Q_{1,2r}-Q_{1,r})(rx)\|_{L^{\infty}(B_{1})}+r^{a_-+1}\|(Q_{2,2r}-Q_{2,r})(rx)\|_{L^{\infty}(B_{1})}\\
\leq&C\left[r^{a_+}\left\|(Q_{1,2r}-Q_{1,r})(rx)(d(rx)/r)^{a_+}\right\|_{L^{\infty}(B_{2})}
+r^{a_-+1}\left\|(Q_{2,2r}-Q_{2,r})(rx)(d(rx)/r)^{a_-+1}\right\|_{L^{\infty}(B_{2})}\right]\\
=&C\left[\left\|(Q_{1,2r}-Q_{1,r})(rx)d^{a_+}(rx)\right\|_{L^{\infty}(B_{2})}
+\left\|(Q_{2,2r}-Q_{2,r})(rx)d^{a_-+1}(rx)\right\|_{L^{\infty}(B_{2})}\right]\\
\leq&C\left[\left\|(Q_{1,2r}-Q_{1,r})d^{a_+}+(Q_{2,2r}-Q_{2,r})d^{a_-+1}\right\|_{L^{\infty}(B_{2r})}\right]\\
\leq &C\zeta(r)r^{k+\delta+a_-}
\end{split}\end{equation}
Therefore, we deduce that
\[\sum_{|\mu|\leq k_1}|q_{1,r}^{(\mu)}-q_{1,2r}^{(\mu)}|\leq C\zeta(r)r^{k+\delta+a_--a_+-|\mu|}\quad\text{and}\quad \sum_{|\mu|\leq k_2}|q_{2,r}^{(\mu)}-q_{2,2r}^{(\mu)}|\leq C\zeta(r)r^{k+\delta-|\mu|-1}.\]
It follows that
\[\sum_{|\mu|\leq k_1}|q_{1,r}^{(\mu)}-q_{1,2^{j}r}^{(\mu)}|\leq\sum_{\iota=0}^{j-1}\sum_{|\mu|\leq k_1}|q_{1,2^{\iota}r}^{(\mu)}-q_{1,2^{\iota+1}r}^{(\mu)}|\leq C\sum_{\iota=0}^{j-1}\zeta(2^{\iota}r)(2^{\iota}r)^{k+\delta+a_--a_+-|\mu|}\]
for $|\mu|\leq \lfloor k+\delta+a_--a_+\rfloor$ and that
\[\sum_{|\mu|\leq k_2}|q_{2,r}^{(\mu)}-q_{2,2^jr}^{(\mu)}|\leq C\sum_{\iota=0}^{j-1}\zeta(2^{\iota}r)(2^{\iota}r)^{k+\delta-|\mu|-1}.\]
for $|\mu|\leq\lfloor k+\delta\rfloor-1$. If we take $j$ such that $2^{i}r\in(\frac{\rho}{4},\frac{\rho}{2})$, by the monotonicity of $\zeta$, we get
\[\begin{split}
\sum_{|\mu|\leq k_1}|q_{1,r}^{(\mu)}-q_{1,2^{j}r}^{(\mu)}|\leq& C\sum_{\iota=0}^{j-1}\zeta(2^{\iota}r)(2^{\iota}r)^{k+\delta+a_--a_+-|\mu|}\\
\leq& C\zeta(r)r^{k+\delta+a_--a_+-|\mu|}\sum_{\iota=0}^{j-1}(2^{\iota})^{k+\delta+a_--a_+-|\mu|}\\
\leq& C\zeta(r)(2^{j}r)^{k+\delta+a_--a_+-|\mu|}.
\end{split}\]
Similarly, we have
\[\sum_{|\mu|\leq k_2}|q_{2,r}^{(\mu)}-q_{2,2^jr}^{(\mu)}|\leq C\zeta(r)(2^{j}r)^{k+\delta-|\mu|-1}.\]
Therefore, for any $R>1$, we have
\begin{equation}\label{e3.40}\|(Q_{1,r}-Q_{1,Rr})d^{a_+}\|_{L^{\infty}(B_{Rr})}=\|\sum_{|\mu|\leq k_1}(q_{1,r}^{(\mu)}-q_{1,Rr}^{(\mu)})x^{\mu}d^{a_+}\|_{L^{\infty}(B_{Rr})}
\leq C\zeta(r)(Rr)^{k+\delta+a_-}\end{equation}
and
\begin{equation}\label{e3.41}\|(Q_{2,r}-Q_{2,Rr})d^{a_-+1}\|_{L^{\infty}(B_{Rr})}=\|\sum_{|\mu|\leq k_2}(q_{2,r}^{(\mu)}-q_{2,Rr}^{(\mu)})x^{\mu}d^{a_-+1}\|_{L^{\infty}(B_{Rr})}
\leq C\zeta(r)(Rr)^{k+\delta+a_-}.\end{equation}
Hence, using the triangular inequality together with \eqref{e3.40} and \eqref{e3.41}, we obtain
\[\begin{split}
&\|v-Q_{1,r}d^{a_+}-Q_{2,r}d^{a_-+1}\|_{L^{\infty}(B_{Rr})}\\
\leq&\|v-Q_{1,Rr}d^{a_+}-Q_{2,Rr}d^{a_-+1}\|_{L^{\infty}(B_{Rr})}+\|(Q_{1,r}-Q_{1,Rr})d^{a_+}\|_{L^{\infty}(B_{Rr})}\\
&+\|(Q_{2,r}-Q_{2,Rr})d^{a_-+1}\|_{L^{\infty}(B_{Rr})}\\
%\leq&C\zeta(Rr)(Rr)^{k+\delta+a_-}+ C\zeta(r)(Rr)^{k+\delta+a_-}\\
\leq& C\zeta(r)(Rr)^{k+\delta+a_-}.
\end{split}\]
\end{proof}

We also need the following technical lemma, which will be applied to $u-c_0d^{\frac{2}{2+\gamma}}$ in the next section.
\begin{lemma}\label{lemma3.7}
Let $1<k\in\mathbb{N}$, $\delta\in(0,1)$ and let $\Omega\subset\mathbb{R}^n$ be such that $\partial\Omega\cap B_1$ is a $C^{k,\delta}$ surface. Let us assume that for any $z\in\Omega\cap B_1$, there exists $Q_1\in\mathbf{P}_{\lfloor k+\delta-a_++a_-\rfloor-1}$ and $Q_2\in\mathbf{P}_{\lfloor k+\delta\rfloor-1}$ such that for any $y_0\in B_{1/2}(z)\cap\Omega$ with $\mathrm{dist}(y_0,\partial\Omega)=|y_0-z|=2r$, the function $v\in C^{k,\delta}(\Omega\cap B_1)$ satisfies
\begin{equation}\label{e3.59}
[v-Q_1 d^{a_+}-Q_2 d^{a_-+1}]_{C^{k-1,\delta}(\overline{B_r}(y_0))}\leq Cr^{a_-+1}
\end{equation}
and
\begin{equation}\label{e3.42}
\|v-Q_1 d^{a_+}-Q_2 d^{a_-+1}\|_{L^{\infty}(B_r(y_0))}\leq Cr^{k+\delta+a_-}.
\end{equation}
Then,
\begin{equation}\label{e3.43}
\left\|\frac{v}{d^{a_-+1}}\right\|_{C^{k-1,\delta}(\overline{\Omega}\cap B_{1/2})}\leq C.
\end{equation}
Moreover, if there is an integer $m>k+\delta-1$ such that $v\in C^{m}(\Omega\cap B_1)$ and for any $\lambda\in\mathbb{N}^n$ with $|\lambda|=m$, it holds
\begin{equation}\label{e3.44}
\left\|\partial^{\lambda}(v-Q_1 d^{a_+}-Q_2 d^{a_-+1})\right\|_{L^{\infty}(B_r(y_0))}\leq C_m r^{k+\delta+a_--m},
\end{equation}
then
\begin{equation}\label{e3.45}
\left|\partial^{\lambda}\left(\frac{v}{d^{a_-+1}}\right)\right|\leq C_m r^{k+\delta-1-m}.
\end{equation}
\end{lemma}

\begin{proof}
Let $w=v-Q_1 d^{a_+}-Q_2 d^{a_-+1}$, then
\[d^{-(a_-+1)}v=wd^{-(a_-+1)}+Q_1 d^{a_+-(a_-+1)}+Q_2.\]
Therefore, for any $x_1,x_2\in B_r(y_0)$, we have
\begin{equation}\label{e3.46}\begin{split}
&\partial^{k-1}\left(d^{-(a_-+1)}v\right)(x_1)-\partial^{k-1}\left(d^{-(a_-+1)}v\right)(x_2)\\
=&\left[\partial^{k-1}\left(wd^{-(a_-+1)}\right)(x_1)-\partial^{k-1}\left(wd^{-(a_-+1)}\right)(x_2)\right]+\left[\partial^{k-1}Q_2(x_1)-\partial^{k-1}Q_2(x_2)\right]
\\
&+\left[\partial^{k-1}\left(Q_1 d^{a_+-(a_-+1)}\right)(x_1)-\partial^{k-1}\left(Q_1 d^{a_+-(a_-+1)}\right)(x_2)\right],
\end{split}\end{equation}
where
\begin{equation}\label{e3.51}\begin{split}
&\partial^{k-1}\left(wd^{-(a_-+1)}\right)(x_1)-\partial^{k-1}\left(wd^{-(a_-+1)}\right)(x_2)\\
=&\sum_{|\mu|\leq k-1}\left[(\partial^{\mu}w)(x_1)(\partial^{\lambda-\mu}d^{-(a_-+1)})(x_1)-(\partial^{\mu}w)(x_2)(\partial^{\lambda-\mu}d^{-(a_-+1)})(x_2)\right]\\
=& \sum_{|\mu|\leq k-1} [\partial^{\mu}w(x_1)-\partial^{\mu}w(x_2)]\left(\partial^{k-1-|\mu|}d^{-(a_-+1)}\right)(x_1)\\
&+\partial^{\mu}w(x_2)\left[(\partial^{k-1-|\mu|}d^{-(a_-+1)})(x_1)-(\partial^{k-1-|\mu|}d^{-(a_-+1)})(x_2)\right].
\end{split}\end{equation}
Recalling from Lemma \ref{lemma2.3}, we have
\begin{equation}\label{e3.47}
|\partial^{k-1-|\mu|}d^{-(a_-+1)}(x_1)|\leq Cr^{-(a_-+1)-k+1+|\mu|}=Cr^{-a_--k+|\mu|}
\end{equation}
for any $\mu\leq\lambda$. Combining this with the mean value theorem, we deduce that
\begin{equation}\label{e3.48}
|\partial^{k-1-|\mu|}d^{-(a_-+1)}(x_1)-\partial^{k-1-|\mu|}d^{-(a_-+1)}(x_2)|\leq Cr^{-a_--k-\delta+|\mu|}|x_1-x_2|^{\delta}.
\end{equation}
On the other hand, we can use interpolation of Lemma 6.32 in \cite{GT} to deduce from \eqref{e3.59} and \eqref{e3.42} that
\begin{equation}\label{e3.49}
|\partial^{\mu}w(x_2)|\leq Cr^{a_-+k+\delta-|\mu|}.
\end{equation}
Using again the mean value theorem, we obtain
\begin{equation}\label{e3.50}
|\partial^{\mu}w(x_1)-\partial^{\mu}w(x_2)|\leq Cr^{a_-+k-|\mu|}|x_1-x_2|^{\delta}.
\end{equation}
Plugging \eqref{e3.47}-\eqref{e3.50} into \eqref{e3.51}, we get
\begin{equation}\label{e3.52}\left|\partial^{k-1}\left(wd^{-(a_-+1)}\right)(x_1)-\partial^{k-1}\left(wd^{-(a_-+1)}\right)(x_2)\right|\leq C|x_1-x_2|^{\delta}.\end{equation}
Since $Q_1\in\mathbf{P}_{\lfloor k+\delta-a_++a_-\rfloor-1}$ and $d\in C^{k,\delta}$, we have
\begin{equation}\label{e3.53}
\left|\partial^{k-1}\left(Q_1 d^{a_+-(a_-+1)}\right)(x_1)-\partial^{k-1}\left(Q_1 d^{a_+-(a_-+1)}\right)(x_2)\right|\leq|x_1-x_2|^{\delta},
\end{equation}
and using again the mean value theorem for the polynomials $Q_2\in\mathbf{P}_{\lfloor k+\delta\rfloor-1}$, we deduce that
\begin{equation}\label{e3.54}
\left|\partial^{k-1}Q_2(x_1)-\partial^{k-1}Q_2(x_2)\right|\leq C |x_1-x_2|^{\delta}.
\end{equation}
Combining with \eqref{e3.46}, \eqref{e3.52}-\eqref{e3.54}, we conclude that
\begin{equation}\label{e3.55}
\left|\partial^{k-1}\left(d^{-(a_-+1)}v\right)(x_1)-\partial^{k-1}\left(d^{-(a_-+1)}v\right)(x_2)\right|\leq C|x_1-x_2|^{\delta},
\end{equation}
which implies that \eqref{e3.43} holds.

Let us assume now \eqref{e3.44} holds. We can proceed as in \eqref{e3.46} for $|\lambda|=m$ to deduce that
\begin{equation}\label{e3.56}
\begin{split}
\partial^{\lambda}(d^{-(a_-+1)}v)=&\partial^{\lambda}(d^{-(a_-+1)}v-d^{a_+-(a_-+1)}Q_1-Q_2)(x)\\
=&\sum_{|\mu|\leq m}\partial^{\mu}(v-d^{a_+}Q_1-d^{a_-+1}Q_2)\ \partial^{\lambda-\mu}d^{-(a_-+1)}.
\end{split}
\end{equation}
As in \eqref{e3.47}, we have
\begin{equation}\label{e3.57}\left|\partial^{\lambda-\mu}d^{-(a_-+1)}\right|\leq C r^{-(a_-+1)-m+|\mu|}.\end{equation}
We can interpolate \eqref{e3.44} by using again Lemma 6.32 in \cite{GT} to deduce that
\begin{equation}\label{e3.58}
\sum_{|\mu|\leq m}r^{|\mu|}\|\partial^{\mu}(v-d^{a_+}Q_1-d^{a_-+1}Q_2)\|_{L^{\infty}(B_r(y_0))}\leq Cr^{k+\delta+a_-}.
\end{equation}
Altogether with \eqref{e3.56}, \eqref{e3.57} and \eqref{e3.58}, we conclude that
\[\left|\partial^{\lambda}(d^{-(a_-+1)}v)\right|\leq Cr^{k+\delta-m-1}\leq C d^{k+\delta-m-1},\]
where we used the comparability of $d(y_0)$ and $r$ in the last inequality.
\end{proof}

%We finish this section with a corollary of Lemma \ref{lemma3.7} that will be applied to $u-c_0d^{\frac{2}{2+\gamma}}$ in the next section.
%\begin{lemma}\label{lemma3.8}
%Let $\Omega$ be a $C^{k,\delta}$ domain with $1<k\in\mathbb{N}$, $\delta\in(0,1)$. Let $v\in C(\Omega\cap B_1)$ satisfies the following equation
%\begin{equation}\label{e3.59}
%-\Delta v+\kappa\frac{v}{d^2}=fd^{\mu-1}\quad\text{in }\Omega\cap B_1,
%\end{equation}
%where $f\in C^{k-2,\delta}(\overline{\Omega}\cap B_1)$. If it satisfies the boundary condition
%\begin{equation}\label{e3.60}
%\lim_{x\rightarrow\partial\Omega}\frac{v}{d^{a_+}}=0,
%\end{equation}
%then, $v\in C^{k-2,\delta}(\overline{\Omega}\cap B_{1/2})$ and
%\[\left\|\frac{v}{d^{a_-+2}}\right\|_{C^{k-2,\delta}(\overline{\Omega}\cap B_{1/2})}
%\leq C\left(\left\|f\right\|_{C^{k-2,\delta}(\overline{\Omega}\cap B_{1})}+\left\|v\right\|_{L^{\infty}(\overline{\Omega}\cap B_{1})}\right).\]

%Moreover, if there is $m>k-2+\delta$ such that $v\in C^{m}(\Omega\cap B_1)$, $f\in  C^{m}(\Omega\cap B_1)$ and
%\begin{equation}\label{e3.61}
%|\partial^{\mu}f|\leq C_{\mu}d^{k-2+\delta-|\mu|}\quad\text{in }\Omega\cap B_1
%\end{equation}
%for any $|\mu|\in\{k-1,\cdots,m-2\}$, then
%\begin{equation}\label{e3.62}
%\left|\partial^{\lambda}\left(\frac{v}{d^{a_-+2}}\right)\right|\leq C_{\mu}d^{k-2+\delta-|\lambda|}\quad\text{in }\Omega\cap B_{1/2}
%\end{equation}
%for any $|\lambda|\in\{k-1,\cdots,m\}$.
%\end{lemma}

\vspace{-0.1cm}

\section{Proof of Theorem \ref{theorem1.1}}

This section is devoted to the proof of Theorem \ref{theorem1.1}. We begin by proving a crucial proposition, which will then be iterated to complete the proof.
\begin{proposition}\label{proposition4.1}
Let $\gamma\in(0,2), k\in\mathbb{N}, \delta\in(0,1)$ and let $\Omega\subset\mathbb{R}^n$ be a $C^{k,\delta}$ domain such that $0\leq u\in C^{\frac{2}{2+\gamma}}$ be a nontrivial solution of
\begin{equation}\label{e4.1}
\begin{cases}
\Delta u=-u^{-\gamma-1}&\text{in }\Omega\cap B_1\\
u=0&\text{on }\partial\Omega\cap B_1.
\end{cases}\end{equation}
Then, we have
\begin{equation}\label{e4.2}
\frac{u}{d^{\frac{2}{2+\gamma}}}\in C^{k-1,\delta}(\overline{\Omega}\cap B_r)\quad\text{and}\quad u_i d^{\frac{\gamma}{2+\gamma}}\in C^{k-1,\delta}(\overline{\Omega}\cap B_r)
\end{equation}
for any $r<1$.
\end{proposition}
\begin{proof}{\em Step 1.} Let $v=u-c_0d^{\frac{2}{2+\gamma}},$ where $c_0$ is defined as in \eqref{e1.3}. We show that $v$ satisfies \[L_{\kappa}v=h d^{\frac{2}{2+\gamma}-2}\] for some function $h$ whose regularity will be improved later.

 \medskip
 Since $u$ is a solution of problem \eqref{e4.1}, by a direct computation, we have
\begin{equation}\label{e4.3}
\begin{split}
\Delta v=&\Delta u-\Delta\left(c_0d^{\frac{2}{2+\gamma}}\right)\\
=&-u^{-\gamma-1}+\left(c_0d^{\frac{2}{2+\gamma}}\right)^{-\gamma-1}-gd^{\frac{2}{2+\gamma}-2},
\end{split}
\end{equation}
where
\[gd^{\frac{2}{2+\gamma}-2}=\Delta\left(c_0d^{\frac{2}{2+\gamma}}\right)+\left(c_0d^{\frac{2}{2+\gamma}}\right)^{-\gamma-1}
=c_0\left[\Delta\left(d^{\frac{2}{2+\gamma}}\right)-\frac{2}{2+\gamma}\left(\frac{2}{2+\gamma}-1\right)d^{\frac{2}{2+\gamma}-2}\right].\]
Notice that by Lemma \ref{lemma2.3}, we have that $|g|\leq Cd^{k+\delta-1}$.
It follows form the Taylor's Theorem that
\[f(p)-f(q)=f'(q)(p-q)+(p-q)^2\int_0^1(1-\tau)(\gamma+1)(\gamma+2)[q+\tau(p-q)]^{-\gamma-3}d\tau.\]
Applying the Taylor's expansion to $f(t)=t^{-\gamma-1}-1$ with $p=u$ and $q=c_0d^{\frac{2}{2+\gamma}}$, we deduce that
\begin{equation}\label{e4.4}
u^{-\gamma-1}-\left(c_0d^{\frac{2}{2+\gamma}}\right)^{-\gamma-1}=(-\gamma-1)\left(c_0 d^{\frac{2}{2+\gamma}}\right)^{-\gamma-2}\left(u-c_0 d^{\frac{2}{2+\gamma}}\right)+fd^{\frac{2}{2+\gamma}-2},
\end{equation}
where
\begin{equation}\label{e4.5}
f=\left(\frac{u}{ d^{\frac{2}{2+\gamma}}}-c_0\right)^2\int_0^1(\gamma+1)(\gamma+2)(1-\tau)\left(c_0(1-\tau)+  \frac{u}{ d^{\frac{2}{2+\gamma}}}\tau\right)^{-\gamma-3}d\tau.
\end{equation}
Since $(-\gamma-1)c_0^{-\gamma-2}=-\kappa$, together with \eqref{e4.3} and \eqref{e4.4}, we obtain
\begin{equation}\label{e4.6}
L_{\kappa}v=(f+g)d^{\frac{2}{2+\gamma}-2}.
\end{equation}

\medskip
{\em Step 2.} We show that for $k=1$,
\begin{equation}\label{e4.7}
\frac{u}{d^{\frac{2}{2+\gamma}}}\in C^{0,\delta}\quad\text{and}\quad u_i d^{\frac{\gamma}{2+\gamma}}\in C^{0,\delta},
\end{equation}
where $u_i=\frac{\partial u}{\partial x_i}$ for all $i=1,\cdots,n$.

From Proposition \ref{proposition2.3}, we can apply the Lemma \ref{lemma3.7} with $Q_1=0$ and $k=1$ to deduce that $|u/d^{\frac{2}{2+\gamma}}|\leq C d^{\delta}$. Thus, we only need to show that $u_i d^{\frac{\gamma}{2+\gamma}}\in C^{0,\delta}$.
Given any $z\in\partial\Omega\cap B_1$ and any $y_0\in B_{1/2}\cap\Omega$ such that $\mathrm{dist}(y_0)=|y_0-z|=2r$, we consider
\[v_r(x)=r^{-\frac{2}{2+\gamma}-\delta}v(y_0+rx).\]
Therefore,
\[\begin{split}
L_{\kappa}v_r(x)=&r^{2-\frac{2}{2+\gamma}-\delta}(L_{\kappa}v)(y_0+rx)\\
=&\left(\frac{d(y_0+rx)}{r}\right)^{\frac{2}{2+\gamma}-2} r^{-\delta}[f(y_0+rx)+g(y_0+rx)],
\end{split}\]
which implies that
\begin{equation}\label{e4.8}
\frac{-\Delta v_r}{\left(\frac{d(y_0+rx)}{r}\right)^{\frac{2}{2+\gamma}-2}}
+\frac{\kappa}{\left(\frac{d(y_0+rx)}{r}\right)^2}\cdot\frac{v_r}{\left(\frac{d(y_0+rx)}{r}\right)^{\frac{2}{2+\gamma}-2}}=r^{-\delta}[f(y_0+rx)+g(y_0+rx)].
\end{equation}
Notice that $f$ has at least the same regularity up to the boundary as $u/d^{\frac{2}{2+\gamma}}$, we have that $|f|\leq Cd^{\delta}$. This combined with the regularity of $g$ that $|g|\leq Cd^{\delta}$ for $k=1$ yields that the right-hand-side of \eqref{e4.8} are bounded. Applying the standard gradient estimation, we obtain
\[\|\nabla v_r\|_{L^{\infty}(B_{1/2})}\leq C\left[\|v_r\|_{L^{\infty}(B_1)}+r^{-\delta}\|f(y_0+rx)+g(y_0+rx)\|_{L^{\infty}(B_1)}\right]\leq C,\]
which implies that
\[\|\nabla v\|_{L^{\infty}(B_{\frac{r}{2}}(y_0))}\leq C r^{\frac{2}{2+\gamma}+\delta-1},\]
namely,
\[\|D(u-c_0 d^{\frac{2}{2+\gamma}})\|_{L^{\infty}(B_{\frac{r}{2}}(y_0))}\leq C r^{\frac{2}{2+\gamma}+\delta-1}.\]
Using Lemma \ref{lemma3.7} with $Q_1=0$ and $m=1$, we deduce that
\begin{equation}\label{e4.9}
\left|D\left(\frac{u}{d^{\frac{2}{2+\gamma}}}\right)\right|\leq Cd^{\delta-1}.
\end{equation}
Therefore, by \eqref{e4.9}, the following identity
\[u_i d^{\frac{\gamma}{2+\gamma}}=d\partial_i\left(\frac{u}{d^{\frac{2}{2+\gamma}}}\right)+\frac{2}{2+\gamma}\partial_i d\cdot\frac{u}{d^{\frac{2}{2+\gamma}}}\]
yields that
\[u_i d^{\frac{2}{2+\gamma}}\in C^{\delta}(\overline{\Omega}\cap B_1).\]

\medskip
{\em Step 3.} For $k>1$, we need to prove that there exists a constant $C>0$ and a radius $\rho>0$ such that for any $x_0\in\partial\Omega\cap B_{\rho}$, there exist polynomials $Q_1\in\mathbf{P}_{\lfloor k+\delta-a_++a_-\rfloor-1}$ and $Q_2\in\mathbf{P}_{\lfloor k+\delta\rfloor-1}$ such that
\begin{equation}\label{e4.10}
\left|v(x)-Q_{1}(x)d^{a_+}-Q_2(x)d^{a_-+1}\right|\leq C|x-x_0|^{k+\delta+a_-}\quad\text{for any }x\in B_{\rho}(x_0)\cap\Omega.
\end{equation}

Let $x_0=0$ for convenience and assume by contradiction that there exists a sequence of functions $v_j$ such that for any sequence of polynomials $P_{1,j}\in\mathbf{P}_{\lfloor k+\delta-a_++a_-\rfloor}-1$ and $P_{2,j}\in\mathbf{P}_{\lfloor k+\delta\rfloor-1}$,
\begin{equation}\label{e4.11}
\sup_{r>0}\frac{1}{r^{k+\delta+a_-}}\|v_{j}-P_{1,j}d^{a_+}-P_{2,j}d^{a_-+1}\|_{L^{\infty}(B_r)}=\infty
\end{equation}
with
\[L_{\kappa}v_j=(f_j+g_j)d^{a_--1}\quad\text{in }\Omega\cap B_2.\]
We select the least squares polynomials $Q_{1,j}$ and $Q_{2,j}$ satisfies the orthogonality condition
\begin{equation}\label{e4.12}
\int_{B_r}\left(v_j-Q_{1,j}d^{a_+}-Q_{2,j}d^{a_-+1}\right)\left(P_{1}d^{a_+}+P_{2}d^{a_-+1}\right)=0
\end{equation}
for any $Q_{1}\in\mathbf{P}_{\lfloor k+\delta-a_++a_-\rfloor}-1$ and $Q_{2}\in\mathbf{P}_{\lfloor k+\delta\rfloor-1}$.

Define
\[\zeta(r)=\frac{1}{r^{k+\delta+a_-}}\|v_{j}-Q_{1,j}d^{a_+}-Q_{2,j}d^{a_-+1}\|_{L^{\infty}(B_r)},\]
then $\lim_{r\rightarrow0}\zeta(r)=\infty$ and there exists a sequence of $\{r_j\}$ with $r_j\rightarrow0$ as $j\rightarrow\infty$ such that
\begin{equation}\label{e4.13}
w_j(x):=\frac{v_{j}(r_jx)-Q_{1,j}(r_jx)d^{a_+}(r_jx)-Q_{2,j}(r_jx)d^{a_-+1}(r_jx)}{\zeta(r_j)r_j^{k+\delta+a_-}}
\end{equation}
satisfies $\|w_j\|_{L^{\infty}(B_1)}\in[1/2,1]$.

\medskip
We now consider the equation satisfied by $L_{\kappa}w_j$. We first estimate the Laplacian of $Q_{1,j}d^{a_+}+Q_{2,j}d^{a_-+1}$ as follows
\begin{equation}\label{e4.14}\begin{split}
&\frac{1}{\zeta(r_j)r_j^{k+\delta+a_-}}\Delta_{x}\left(Q_{1,j}(r_jx)d^{a_+}(r_jx)+ Q_{2,j}(r_jx)d^{a_-+1}(r_jx)\right)\\
=&\frac{r_j^2}{\zeta(r_j)r_j^{k+\delta+a_-}}\left[d^{a_+} \Delta Q_{1,j} +d^{a_-+1} \Delta Q_{2,j} +2a_+d^{a_+-1}\nabla Q_{1,j} \cdot\nabla d \right.\\
&\left.+2(a_-+1)d^{a_-}\nabla  Q_{2,j} \cdot\nabla d +Q_{1,j} \Delta(d^{a_+} )+ Q_{2,j} \Delta(d^{a_-+1})\right](r_j x)
\end{split}\end{equation}
Since $\|\Omega\cap B_2\|_{C^{k,\delta}}\leq 1$, we have
\begin{equation}\label{e4.15}d(x)=P(x)+g(x),\end{equation}
where $P\in\mathbf{P}_{k}$ with $\|P\|_{L^{\infty}(\Omega\cap B_1)}\leq C$ and $g\in C^{\infty}(\Omega\cap B_2)$ with $|\nabla g|\leq C|x|^{k+\delta-1}$. Thanks to Lemma \ref{lemma2.3}, we have
\[\Delta(d^{a_+} )=\kappa d^{a_+-2} +d^{a_+-1} g_2 \]
and
\[\Delta(d^{a_-+1} )=(\kappa+2a_-)d^{a_--1} +d^{a_-} g_2\]
with $g_2\in C^{k+\delta-2}(\overline{\Omega}\cap B_1)\cap C^{\infty}(\Omega\cap B_1)$.
Therefore,
\begin{equation}\label{e4.16}\begin{split}
&\frac{r_j^2}{\zeta(r_j)r_j^{k+\delta+a_-}}\left[Q_{1,j} \Delta(d^{a_+} )+Q_{2,j} \Delta(d^{a_-+1} )-Q_{1,j} \kappa d^{a_+-2} -Q_{2,j} \kappa d^{a_--1}\right](r_jx)\\
=&\frac{r_j^2}{\zeta(r_j)r_j^{k+\delta+a_-}}\left[Q_{1,j}d^{a_+-1} g_2+Q_{2,j}(2a_-d^{a_--1} +d^{a_-} g_2 ) \right](r_jx)\\
%=&\frac{2a_-r_j^3}{\zeta(r_j)r_j^{k+\delta}}\left(\frac{d(r_j x)}{r_j}\right)^{a_--1}+m_{1,r_j}(x)
=&R_{1,r_j}(x)\left(\frac{d(r_j x)}{r_j}\right)^{a_--1}+m_{1,r_j}(x),
\end{split}\end{equation}
where $R_{1,r_j}\in\mathbf{P}_{M}$ for some $M\in\mathbb{N}$ and $|m_{1,r_j}(x)|\leq Cm(r)$ such that $m(r)\rightarrow0$ as $r\rightarrow0$.
Using \eqref{e4.15}, we deduce that
\begin{equation}\label{e4.17}\frac{r_j^2}{\zeta(r_j)r_j^{k+\delta+a_-}}\left[d^{a_+}\Delta Q_{1,j}+d^{a_-+1}\Delta Q_{2,j}\right](r_jx)=R_{2,r}(x)\left(\frac{d(r_jx)}{r_j}\right)^{a_+-1}+R_{3,r}(x)\left(\frac{d(r_jx)}{r_j}\right)^{a_-}+m_{2,r_j}(x),\end{equation}
where $R_{2,r_j},R_{3,r_j}\in\mathbf{P}_{M}$ and $|m_{2,r_j}(x)|\leq Cm(r)$.
Using again \eqref{e4.15}, we have

\begin{equation}\label{e4.18}\begin{split}
 \frac{r_j^2d^{a_+-1}}{\zeta(r_j)r_j^{k+\delta+a_-}}(\nabla Q_{1,j}\cdot\nabla d)(r_jx)
=&\frac{r_j^2d^{a_+-1}}{\zeta(r_j)r_j^{k+\delta+a_-}}(\nabla Q_{1,j}\cdot\nabla P+\nabla Q_{1,j}\cdot\nabla g)(r_jx)\\
=&R_{4,r_j}(x)\left(\frac{d(r_jx)}{r_j}\right)^{a_+-1}+m_{3,r_j}(x)
\end{split}\end{equation}
and
\begin{equation}\label{e6.1}\begin{split}
 \frac{r_j^2d^{a_-}}{\zeta(r_j)r_j^{k+\delta+a_-}}(\nabla Q_{2,j}\cdot\nabla d)(r_jx)
=&\frac{r_j^2d^{a_-}}{\zeta(r_j)r_j^{k+\delta+a_-}}(\nabla Q_{2,j}\cdot\nabla P+\nabla Q_{2,j}\cdot\nabla g)(r_jx)\\
=&R_{5,r_j}(x)\left(\frac{d(r_jx)}{r_j}\right)^{a_-}+m_{4,r_j}(x)
\end{split}\end{equation}
for some polynomials $R_{4,r_j},R_{5,r_j}\in\mathbf{P}_{M}$ and $|m_{3,r_j}(x)|\leq Cm(r)$, $|m_{4,r_j}(x)|\leq Cm(r)$.
Altogether with \eqref{e4.14}-\eqref{e6.1}, we derive that
\begin{equation}\label{e4.19}
L_{\kappa}(w_j)=\widetilde{Q}_{1,r_j}\left(\frac{d(r_jx)}{r_j}\right)^{a_+-1}+\widetilde{Q}_{2,r_j}\left(\frac{d(r_jx)}{r_j}\right)^{a_--1}+m_{r_j}(x),
\end{equation}
where $\widetilde{Q}_{1,r_j},\widetilde{Q}_{2,r_j}\in\mathbf{P}_{M}$ for some $M\in\mathbb{N}$ and $m_{r_j}(x)\rightarrow0$ as $j\rightarrow\infty$.
Moreover, Lemma \ref{lemma3.6} provides the growth bound
\begin{equation}\label{e4.20}
\|w_j\|_{L^{\infty}(B_R)}\leq CR^{k+\delta+a_-}
\end{equation}
for any $R>1$. On the other hand,  in virtue of Lemma \ref{lemma3.4}, the sequence of $\widetilde{Q}_{1,r_j},\widetilde{Q}_{2,r_j}$ are uniformly bounded in $j$, and then there exists some function $w_0$ such that $w_j\rightarrow w_0$ as $j\rightarrow\infty$. Then, by \eqref{e4.19} and \eqref{e4.20}, $w_0$ satisfies
\begin{equation}\label{e4.21}
\begin{cases}
L_{\kappa}w_0=\widetilde{Q}_{1}(x)\left(x_n\right)_+^{a_+-1}+\widetilde{Q}_{2}(x)\left(x_n\right)_+^{a_--1}&\text{in }\{x_n>0\}\\
\|w_0\|_{L^{\infty}(B_R)}\leq CR^{k+\delta+a_-}&
\end{cases}
\end{equation}
with $\widetilde{Q}_{1},\widetilde{Q}_{2}\in\mathbf{P}_{M}$ for some $M\in\mathbb{N}$. Additionally, the uniform convergence of $w_j$ and the inequality \eqref{e3.17} imply that
\begin{equation}\label{e4.22}
\|w_0\|\leq C (x_n)_+^{\frac{1}{2}+a_-}.
\end{equation}
Therefore,  Lemma \ref{lemma3.5} implies that
\[w_0=Q_1(x)x_n^{a_+}+Q_2(x)x_n^{a_-+1}\]
with $Q_1\in\mathbf{P}_{\lfloor k+\delta-a_++a_-\rfloor-1}$ and  $Q_2\in\mathbf{P}_{\lfloor k+\delta\rfloor-1}$.
Taking limits in the orthogonality condition \eqref{e4.12} as $j\rightarrow\infty$, we have
\[\int_{B_1}w_0\left(P_{1}d^{a_+}+P_{2}d^{a_-+1}\right)=0\]
for any $P_1\in\mathbf{P}_{\lfloor k+\delta-a_++a_-\rfloor-1}$ and  $P_2\in\mathbf{P}_{\lfloor k+\delta\rfloor-1}$. Thus, $w_0=0$, which is a contradiction with $\|w_0\|_{L^{\infty}(B_1)}\geq\frac{1}{2}$.

\medskip
{\em Step 4.} We now prove that $f\in C^{k-1,\delta}(\overline{\Omega}\cap B_r)$ for all $r<1$ and $k>1$.

By inductive hypothesis, we assume that
\[\frac{u}{d^{\frac{2}{2+\gamma}}}\in C^{k-2,\delta}(\overline{\Omega}\cap B_1).\]
Thanks to Lemma \ref{lemma2.3}, we have
\begin{equation}\label{e4.23}
g\in C^{k-1,\delta}(\overline{\Omega}\cap B_1),
\end{equation}
whereas $f\in C^{k-2,\delta}$ due to \eqref{e4.5}.

We first claim that
\begin{equation}\label{e4.24}
\left|D^{m}\left(\frac{u}{d^{\frac{2}{2+\gamma}}}\right)\right|\leq C d^{k+\delta-2-m}\quad\text{in }\Omega\cap B_r,
\end{equation}
for any $m<k+\delta-2$. Assume by inductive hypothesis that \eqref{e4.24} holds for any $m\leq l$, we need to show that \eqref{e4.24} holds for $m=l+1$. Arguing as Step 3, we know that for any $x_0\in\partial\Omega\cap B_1$, there exists polynomials $Q_1\in\mathbf{P}_{\lfloor k+\delta-a_++a_-\rfloor-2}$ and $Q_2\in\mathbf{P}_{\lfloor k+\delta\rfloor-2}$ such that
\begin{equation}\label{e4.25}
\left|v(x)-Q_{1}(x)d^{a_+}-Q_2(x)d^{a_-+1}\right|\leq C|x-x_0|^{k+\delta+a_--1}\quad\text{for any }x\in B_{\rho}(x_0)\cap\Omega.
\end{equation}
For any $y_0\in\Omega\cap B_{1/2}(x_0)$ such that $\mathrm{dist}(y_0)=|y_0-x_0|=2r$, we define the rescaling
\[v_r=\frac{v(rx)-Q_1(rx)d^{a_+}(rx)-Q_2(rx)d^{a_-+1}(rx)}{r^{k+\delta+a_--1}},\]
we calculate that
\begin{equation}\label{e4.26}\begin{split}
L_{\kappa}v_r=&r^{3-k-\delta}\left[\frac{1}{r}(f(y_0+rx)+g(y_0+rx))\left(\frac{d(y_0+rx)}{r}\right)^{a_--1}\right.\\
&\left.-r^{a_+-a_-}L_{\kappa}\left(Q_1(y_0+rx)\left(\frac{d(y_0+rx)}{r}\right)^{a_+}\right)
-rL_{\kappa}\left(Q_2(y_0+rx)\left(\frac{d(y_0+rx)}{r}\right)^{a_-+1}\right)\right]
\end{split}\end{equation}
Since $f+g=0$ on the boundary, we can write $\frac{1}{d}(f+g) =(f_0+g_0)$ with $(f_0+g_0)\in C^{k-3,\delta}(\overline{\Omega}\cap B_1)$. Therefore, because $d$ is comparable with $r$ near $y_0$, it follows that
\begin{equation}\label{e4.29}\begin{split}
&r^{3-k-\delta}[r^{-1}(f(y_0+rx)+g(y_0+rx))]\left(\frac{d(y_0+rx)}{r}\right)^{a_--1}\\
&=r^{3-k-\delta}[f_0(y_0+rx)+g_0(y_0+rx)]\left(\frac{d(y_0+rx)}{r}\right)^{a_--1}.
\end{split}\end{equation}
 Additionally, by Lemma \ref{lemma2.3}, we have
\begin{equation}\label{e4.27}L_{\kappa}\left(Q_1\left(\frac{d(y_0+rx)}{r}\right)^{a_+}\right)(y_0+rx)=-h_1(y_0+rx)\left(\frac{d(y_0+rx)}{r}\right)^{a_+-1}\end{equation}
and
\begin{equation}\label{e4.28}L_{\kappa}\left(Q_2\left(\frac{d(y_0+rx)}{r}\right)^{a_-+1}\right)(y_0+rx)=-h_2(y_0+rx)\left(\frac{d(y_0+rx)}{r}\right)^{a_--1}\end{equation}
for some $h_1,h_2\in C^{k-2,\delta}(\overline{\Omega}\cap B_1)$.
Altogether with \eqref{e4.26}-\eqref{e4.28}, we deduce that
\[L_{\kappa}v_r=r^{3-k-\delta}h_0(x)\left(\frac{d(y_0+rx)}{r}\right)^{a_--1},\]
where $h_0(y_0+rx)=(f_0+g_0)(y_0+rx)+( h_1d^{a_+-a_-})(y_0+rx)+rh_2(y_0+rx)$ belong to $C^{k-3,\delta}(\overline{\Omega}\cap B_1)$. Then, we obtain
\begin{equation}\label{e4.30}
\frac{-\Delta v_r}{\left(\frac{d(y_0+rx)}{r}\right)^{a_--1}}
+\frac{\kappa}{\left(\frac{d(y_0+rx)}{r}\right)^2}\cdot\frac{v_r}{\left(\frac{d(y_0+rx)}{r}\right)^{a_--1}}=r^{3-k-\delta}h_0(y_0+rx).
\end{equation}
Since
\begin{equation}\label{e4.31}\left|\partial^{l-1}h_0\right|\leq C_{l}d^{k+\delta-(l+2)},\end{equation}
the interior Schauder estimates imply that
\[\left|\partial^{l+1}v_r\right|_{L^{\infty}(B_{\frac{1}{2}})}\leq C\left(|v_r|_{L^{\infty}(B_1)}+r^{3-k-\delta}\|D^{(l+1)-2}h_0(y_0+rx)\|_{L^{\infty}(B_1)}\right).\]
that is
\[\begin{split}
&r^{-k-\delta-a_-+1+(l+1)}\left|\partial^{l+1}(v-Q_1 d^{a_+}-Q_2 d^{a_-+1})\right|_{L^{\infty}(B_{\frac{r}{2}}(y_0))}\\
\leq& C\left(1+r^{(l+1)-k-\delta+1}\|D^{(l+1)-2}h_0\|_{L^{\infty}(B_{\frac{r}{2}}(y_0))}\right)\leq C,
\end{split}\]
where we used \eqref{e4.31} in the last inequality. Therefore, we obtain
\begin{equation}\label{e4.32}
\left|\partial^{l+1}(v-Q_1 d^{a_+}-Q_2 d^{a_-+1})\right|_{L^{\infty}(B_{\frac{r}{2}}(y_0))}\leq C r^{k+\delta+a_--1-(l+1)},
\end{equation}
by Lemma \ref{lemma3.7}, we conclude that
\begin{equation}\label{e4.33}
\left|\partial^{l+1}\left(\frac{v}{d^{a_-+1}}\right)\right|_{L^{\infty}(B_{\frac{r}{2}}(y_0))}\leq Cd^{k+\delta-2-(l+1)}.
\end{equation}

Now, it remains to show that $f\in C^{k+\delta-1}$. For any $z\in\partial\Omega\cap B_r$, by inductive hypothesis for $k>1$, we have
\[\frac{u}{d^{\frac{2}{2+\gamma}}}-c_0=P_{z}(x)+O(|x-z|^{k+\delta-2}) \quad\text{in }\Omega\cap B_1,\]
where $P_z$ is some polynomials satisfying $P_{z}(z)=0$. Since $P_z(x)=O(|x-z|)$, we have
\[\left(P_{z}(x)+O(|x-z|^{k+\delta-2})\right)^2=P_z^2+O(|x-z|^{k+\delta-1}),\]
 and thus
\[\left(\frac{u}{d^{\frac{2}{2+\gamma}}}-c_0\right)^2=P_z^2+O(|x-z|^{k+\delta-1}).\]
For $k=1$, we use Proposition \ref{proposition2.3} to get
\[\frac{u}{d^{\frac{2}{2+\gamma}}}-c_0=O(|x-z|^{\delta}).\]
In both cases, we also have that
\[\int_0^1(\gamma+1)(\gamma+2)(1-\tau)\left(c_0(1-\tau)+\tau\frac{u}{d^{\frac{2}{2+\gamma}}}\right)^{-\gamma-3}d\tau\in C^{k-1,\delta} \]
and hence by \eqref{e4.5},
\[f(x)=Q_z(x)+O(|x-z|^{k+\delta-1})\]
for any $x\in \Omega\cap B_1$ and some polynomial $Q_z$ satisfying $Q_z(z)=0$. This means that $f\in C^{k-1,\delta}$ pointwise on  $\partial\Omega\cap B_r$. Next, we use the interior regularity estimates to deduce that $f\in C^{k-1,\delta}(\overline{\Omega}\cap B_r)$. Indeed, by \eqref{e4.24} and using that $|u/d^{\frac{2}{2+\gamma}}-c_0|\leq Cd$ for $k>1$, we have
\[\begin{split}
\left|D^k\left(\frac{u}{d^{\frac{2}{2+\gamma}}}-c_0\right)^2\right|\leq & C\sum_{i=0}^{k/2}\left|D^i\left(\frac{u}{d^{\frac{2}{2+\gamma}}}-c_0\right)\right|\left|D^{k-i}\left(\frac{u}{d^{\frac{2}{2+\gamma}}}-c_0\right)\right|\\
\leq&Cd^{\delta-2}d+C\sum_{i=0}^{k/2}\left|D^i\left(\frac{u}{d^{\frac{2}{2+\gamma}}}\right)\right|\left|D^{k-i}\left(\frac{u}{d^{\frac{2}{2+\gamma}}}\right)\right|\\
%\leq&Cd^{\delta-1}+C\sum_{i=1}^{k/2}(1+d^{k+\delta-2-k+i})\\
\leq&Cd^{\delta-1}
\end{split}\]
Therefore, we find
\[|D^k f|\leq C d^{\delta-1}\quad\text{in }\Omega\cap B_r,\]
and then
\[f\in C^{k-1,\delta}(\overline{\Omega}\cap B_r) \]
for $k>1$, as we wanted.

\medskip
{\em Step 5.} We conclude that \eqref{e4.2} holds for any $k>1$.

Arguing as Step 3 using \eqref{e4.10}, we get
\begin{equation}\label{e4.24}
\left|D^{m}\left(\frac{u}{d^{\frac{2}{2+\gamma}}}\right)\right|\leq C d^{k+\delta-1-m}\quad\text{in }\Omega\cap B_r,
\end{equation}
for any $m<k+\delta-1$, which implies that
\[\left|D^k\left(\frac{u}{d^{\frac{2}{2+\gamma}}}\right)\right|\leq Cd^{\delta-1}\quad\text{in }\Omega\cap B_r.\]
 We only need to show that
\begin{equation}\label{e4.35}\frac{u}{d^{\frac{2}{2+\gamma}}}\in C^{k-1,\delta}(\overline{\Omega}\cap B_r).\end{equation}
Once \eqref{e4.35} holds, we deduce that
\[D^k\left(u_i d^{\frac{\gamma}{2+\gamma}}\right)=D^k\left(d\partial_i\left(\frac{u}{d^{\frac{2}{2+\gamma}}}\right)+\frac{2}{2+\gamma}\partial_i d\cdot\frac{u}{d^{\frac{2}{2+\gamma}}}\right)\leq Cd^{\delta-1}\quad\text{in }\Omega\cap B_r.\]
This means that
\[u_i d^{\frac{\gamma}{2+\gamma}}\in C^{k-1,\delta}(\overline{\Omega}\cap B_r),\]
 which is the desired result.

In virtue of Lemma \ref{lemma3.7},  \eqref{e4.35} will follow from verifying \eqref{e3.59}.
Thanks to Lemma \ref{lemma2.3}, we have
\begin{equation}\label{e4.23}
g\in C^{k-1,\delta}(\overline{\Omega}\cap B_1),
\end{equation}
whereas $f\in C^{k-1,\delta}$ due to \eqref{e4.5}.
For any $x_0\in\partial\Omega\cap B_1$, we know that there exist  two polynomials such that \eqref{e4.10} holds. Hence, as in the proof of Step 4, for any $y_0\in\Omega\cap B_{1/2}(x_0)$ such that $\mathrm{dist}(y_0)=|y_0-x_0|=2r$, we define
\[v_r=\frac{v(rx)-Q_1(rx)d^{a_+}(rx)-Q_2(rx)d^{a_-+1}(rx)}{r^{k+\delta+a_-}}.\]
Arguing as in Step 4, we calculate that
%\begin{equation}\label{e4.26}\begin{split}
%L_{\kappa}v_r=&r^{2-k-\delta}\left[(f_0(y_0+rx)+g_0(y_0+rx))\left(\frac{d(y_0+rx)}{r}\right)^{a_--1}\right.\\
%&\left.-r^{a_+-a_-}L_{\kappa}\left(Q_1(y_0+rx)\left(\frac{d(y_0+rx)}{r}\right)^{a_+}\right)
%-rL_{\kappa}\left(Q_2(y_0+rx)\left(\frac{d(y_0+rx)}{r}\right)^{a_-+1}\right)\right]，
%\end{split}\end{equation}
%where $(f_0+g_0)=\frac{1}{r}(f+g)\in C^{k-2,\delta}$.
%By Lemma \ref{lemma2.3}, we have
%\begin{equation}\label{e4.27}L_{\kappa}\left(Q_1\left(\frac{d(y_0+rx)}{r}\right)^{a_+}\right)(y_0+rx)=-h_1(y_0+rx)\left(\frac{d(y_0+rx)}{r}\right)^{a_+-1}\end{equation}
%and
%\begin{equation}\label{e4.28}L_{\kappa}\left(Q_2\left(\frac{d(y_0+rx)}{r}\right)^{a_-+1}\right)(y_0+rx)=-h_2(y_0+rx)\left(\frac{d(y_0+rx)}{r}\right)^{a_--1}\end{equation}
%for some $h_1,h_2\in C^{k-2,\delta}(\overline{\Omega}\cap B_1)$.

%Altogether with \eqref{e4.26}-\eqref{e4.28}, we deduce that
\[L_{\kappa}v_r=r^{2-k-\delta}h_0(x)\left(\frac{d(y_0+rx)}{r}\right)^{a_--1},\]
for some $h_0$  belong to $C^{k-2,\delta}(\overline{\Omega}\cap B_1)$, namely,
\begin{equation}\label{e4.30}
\frac{-\Delta v_r}{\left(\frac{d(y_0+rx)}{r}\right)^{a_--1}}
+\frac{\kappa}{\left(\frac{d(y_0+rx)}{r}\right)^2}\cdot\frac{v_r}{\left(\frac{d(y_0+rx)}{r}\right)^{a_--1}}=r^{2-k-\delta}h_0(y_0+rx).
\end{equation}
Since $\frac{d(y_0+rx)}{r}\in C^{k,\delta}(B_1)$ and $\frac{1}{C}\leq \frac{d(y_0+rx)}{r}\leq C$, the standard interior Schauder estimates yields that
\[[v_r]_{C^{k-1,\delta}(B_{1/2})}\leq C\left(\|v_r\|_{L^{\infty}(B_1)}+r^{2-\beta}[h_0]_{C^{k-2,\delta}(B_1)}\right).\]
Rescaling back to $v$, we get
\[r^{-a_--1}[v]_{C^{k-1,\delta}(B_{r/2})}\leq C\left(1+[h_0]_{C^{k-2,\delta}(B_r(y_0))}\right)\leq C,\]
which concludes the proof.

\end{proof}

Finally, we present the proof of Theorem \ref{theorem1.1}.
\begin{proof}[Proof of Theorem \ref{theorem1.1}] Assume that $x_0=0$ and that the inner unit normal vector at $0$ is $\nu(0)=e_n$, and let $\Omega=\{u>0\}$. Since $\Delta u_i=(\gamma+1)u^{-\gamma-2}u_i$ in $\{u>0\}\cap B_1$, it follows that the function $w:=\frac{u_i}{u_n}$ satisfies
\[\mathrm{div}(u_n^2\nabla w)=0\quad\text{in }\Omega\cap B_1.\]
From Proposition \ref{proposition4.1}, we have
\[u_n^2\sim d^{s}\quad\text{and}\quad u_n^2/d^s\in C^{\delta}(B_{r_0\cap\overline{\Omega}})\text{ for any }\delta\in(0,1),\]
where $r_0>0$ and
\[s=-\frac{2\gamma}{2+\gamma}>-1.\]
Therefore, we can rewrite the equation for $w$ as
\begin{equation}\label{e4.34}
\mathrm{div}\left(a(x)d^{s}\nabla w\right)=0\quad\text{in }\Omega\cap B_{r_0},
\end{equation}
where $a(x)\in C^{\delta}(\overline{\Omega}\cap B_{1/2})$ and $\lambda\leq a(x)\leq \Lambda$ for some $\lambda,\Lambda>0$.
Using Proposition \ref{proposition2.3}, we have that
\[w=\frac{u_i}{u_n}\in C^{\delta}(B_{r_0}\cap\overline{\Omega})\]
for any $\delta\in(0,1)$. Combining with this and interior estimates, a standard argument yields $|\nabla w|\leq C d^{\delta-1}$ in $\Omega\cap B_{r_0}$. Applying the Theorem 1.1 in \cite{TTV}, there exists $r\in(0,r_0)$ such that
\[w\in C^{1,\delta}(B_r\cap\overline{\Omega})\]
with boundary condition
\[\nabla w\cdot\nu=0\quad\text{on }\partial\Omega\cap B_{r_0}.\]
The normal vector $\nu$ to $\partial\Omega$ is given by
\[\nu_i=\frac{u_i}{|\nabla u|}=\frac{u_i/u_n}{\sqrt{1+\sum_{j=1}^n(u_j/u_n)^2}},\quad i=1,\cdots,n, \]
yields that $\nu\in C^{1,\delta}$ and $\partial\Omega\in C^{2,\delta}$ (since the normal vector to the free boundary is given by a  $C^{1,\delta}$ function).

Now, since $\Omega$ is $C^{2,\alpha}$, we can apply Proposition \ref{proposition4.1} to obtain that
\[u_i d^{\frac{\gamma}{2+\gamma}}\in C^{1,\delta}(\overline{\Omega}\cap B_r),\]
and thus
\[a(x)=u_n^2/d^s\in C^{1,\delta}(\overline{\Omega}\cap B_r).\]
Using again Theorem 1.1 in \cite{TTV}, we deduce that
\[w\in C^{2,\delta}(\overline{\Omega}\cap B_r)\quad\text{for any }r\in(0,r_0).\]
Iterating this procedure, we find that
\[\partial\Omega\in C^{\infty}\quad\text{and}\quad u_id^{\frac{\gamma}{2+\gamma}}\in C^{\infty}.\]
This yields that $u/d^{\frac{2}{2+\gamma}}\in C^{\infty}$ and thus $u^{\frac{2+\gamma}{2}}=(u/d^{\frac{2}{2+\gamma}})^{\frac{2+\gamma}{2}}d\in C^{\infty}$, and the theorem is proved.

\end{proof}

\subsection*{Acknowledgements}
The authors would like to thank Dr. Daniel Restrepo for his helpful responses to some technical detail about the $C^{\infty}$ regularity for free boundary of Alt-Caffarelli-Phillips functional.
%%%%%%%%

\end{document}